\documentclass[12pt]{amsart}
\usepackage{amsmath,amssymb,amsfonts,amsthm,amstext,graphicx,xcolor,enumerate}
\usepackage{cleveref,url,mathtools,mathdots,verbatim}
\usepackage{enumitem}

\usepackage[margin=3cm]{geometry}

\graphicspath{{pics/}}

\newtheorem{thm}{Theorem}
\newtheorem*{thm*}{Theorem}
\newtheorem{prop}[thm]{Proposition}
\newtheorem{lemma}[thm]{Lemma}

\crefname{thm}{Theorem}{Theorems}
\crefname{lemma}{Lemma}{Lemmas}
\crefname{prop}{Proposition}{Propositions}
\crefname{cor}{Corollary}{Corollaries}
\crefname{section}{Section}{Sections}
\crefname{figure}{Figure}{Figures}
\crefname{question}{Question}{Questions}
\crefname{remark}{Remark}{Remarks}

\renewcommand{\P}{\mathbb P}
\newcommand{\N}{\mathbb N}
\newcommand{\Z}{\mathbb Z}

\newcommand{\thr}{r}

\newcommand{\floor}[1]{\left\lfloor#1\right\rfloor}
\newcommand{\ceil}[1]{\left\lceil#1\right\rceil}

\newcommand{\boxright}{\ensuremath{%
  \mathrel{\square\kern-1.5pt\raise1pt\hbox{$\mathord{\rightarrow}$}}}}
  \newcommand{\nonbrs}{\eta}

\newcommand{\note}[1]{{\textcolor{blue}
{\bf [#1\marginpar{\textcolor{red}{\bf $\bigstar$}}]}}}

\newcommand{\df}[1]{\textbf{#1}}

\title{Polluted Bootstrap Percolation in Three Dimensions}
\date{18 June 2017}
\author{Janko Gravner}
\address{Janko Gravner, Mathematics Department, University of California, Davis, CA 95616}
\email{gravner@math.ucdavis.edu}
\author{Alexander E.~Holroyd}
\address{Alexander E.~Holroyd,
Microsoft Research, Redmond, WA 98052}
\email{holroyd@microsoft.com}
\author{David Sivakoff}
\address{David Sivakoff, Departments of Statistics and Mathematics,
The Ohio State University,
Columbus, OH 43210}
\email{dsivakoff{@}stat.osu.edu}
\begin{document}
\begin{abstract}
In the polluted bootstrap percolation model, vertices of the cubic lattice $\Z^3$ are independently declared initially occupied with probability $p$ or closed with probability $q$.  Under the standard (respectively, modified) bootstrap rule, a vertex becomes occupied at a subsequent step if it is not closed and it has at least $3$ occupied neighbors (respectively, an occupied neighbor in each coordinate). We study the final density of occupied vertices as $p,q\to 0$.  We show that
this density converges to $1$ if $q \ll p^3(\log p^{-1})^{-3}$ for both standard and modified rules.  Our principal result is a complementary
bound with a matching power for the modified model: there exists $C$ such that the final density converges to $0$ if $q > Cp^3$.  For the standard model, we establish convergence to $0$ under the stronger condition $q>Cp^2$.
\end{abstract}
\maketitle

\section{Introduction}


Bootstrap percolation is a fundamental cellular
automaton model for nucleation and growth from sparse
random initial seeds.  In this article we address how the
model is affected by the presence of pollution in the form
of sparse random permanent
obstacles.

Let $\Z^d$ be the set of $d$-vectors of integers, which we
call  \df{vertices} or \df{sites},
and let $p,q\in[0,1]$ be parameters. In the
\df{initial} (time zero) configuration, each vertex is chosen
to have exactly one of three possible states:
$$\begin{cases}
\begin{array}{ll}
  \text{\df{closed}}  & \text{with probability }q; \\
  \text{\df{open} and \df{initially occupied}} & \text{with probability }p; \\
  \text{\df{open} but not initially occupied} & \text{with probability }1-p-q.
\end{array}
\end{cases}
$$
Vertices that are open but not initially occupied are also
called \df{empty}. Initial states are chosen independently
for different vertices. Closed vertices represent pollution
or obstacles, while occupied vertices represent a growing
agent.

The configuration evolves in discrete time steps $t=0,1,2,\ldots$,
and we consider two versions of the bootstrap rule
that determines the evolution.
As usual we make $\Z^d$ into a graph by declaring vertices
$u,v\in\Z^d$ to be neighbors if $\|u-v\|_1=1$. The
\df{threshold} $\thr$ is an integer parameter.  In the \df{standard}
rule, an open site $x$ that is unoccupied at time $t$
becomes occupied at time $t+1$ if and only if
\begin{equation}
\text{at least $\thr$ neighbors of $x$ are occupied}
\label{standard}
\end{equation}
 at time $t$.
In the \df{modified} rule, the condition
\eqref{standard} is replaced with:
\begin{equation}\label{modified}
\text{for at least $\thr$ of the coordinates $i=1,\ldots, d$,
either $x-e_i$ or $x+e_i$ is occupied,}
\end{equation}
where $e_1,\ldots, e_d$ are the standard basis vectors.
In either version, closed vertices remain closed forever,
open vertices remain open, and once a vertex is occupied it remains so
for all later times.


Bootstrap percolation without pollution (the case $q=0$ in
our formulation) has a long and rich history including many
surprises. For $d\geq r \geq 1$, there is no phase
transition in $p$, in the sense that every site of $\Z^d$
is eventually occupied almost surely for every $p>0$, as
proved in \cite{van-enter} ($d=2$) and \cite{schonmann}
($d\geq 3$).  To see a phase transition, one must restrict
the dynamics in some way that is controlled by an
additional parameter. The choice that has received by far
the most attention is restriction to a finite box of large
diameter $n$.  This leads to consideration of metastability
properties of the model, which are by now understood in
great depth (see e.g.\ \cite{AL,Hol1,BBDM,GHM}), as well as
for a broad range of variant growth rules (e.g.\
\cite{GG,DE,BDMS}); for further background see the
excellent recent survey \cite{Mor}. Another natural choice
is to restrict to the complement of a random field of
obstacles of density $q>0$.  This is the subject of the
 current article, together with the recent article~\cite{GH} by two of the current authors. This model, called polluted bootstrap percolation, was
introduced by Gravner and McDonald \cite{GM} in 1997. In
the intervening peroid, rigorous progress on growth
processes in random environments has been limited, but see
\cite{DEKMS,BDGM1,BDGM2,GMa,GZH, JLTV} for some examples of
work on related models.

The principal quantity of interest in polluted bootstrap percolation is the {\it final density\/} of occupied vertices, i.e.\ the probability that the origin is eventually occupied, in the regime where $p$ and $q$ are both small.  In dimension $d=2$ with threshold $\thr=2$,  Gravner and McDonald proved that the final density is strongly dependent
on the relative scaling of $p$ and $q$.  Specifically, for the
standard model,
there exist constants $c,C>0$ such that, as $p\to 0$ and $q\to 0$ simultaneously,
\begin{equation}\label{GM-standard}
\P\bigl(\text{the origin
is eventually occupied}\bigr)\to
\begin{cases}
  1, & \text{if } q<c p^2;\\
  0, & \text{if } q>C p^2.
\end{cases}
\end{equation}
For the modified model, the probability in (\ref{GM-standard})
goes to $1$ under a stronger assumption
$q\ll p^2(\log p^{-1})^{-2}$.
(It is not known whether the logarithmic factor can be reduced
or even eliminated; see \cite{GM} and the second
problem in Section~\ref{open problems}.)

\newcommand{\perc}{\text{\tt perc}_0}

By contrast, when $\thr=2$ and $d\geq 3$, the main result from
\cite{GH} is that
occupation prevails regardless of the $p$ versus $q$ scaling.
We call a set of vertices
\df{connected} if it induces a connected subgraph of $\Z^d$,
and let $\perc$ be the event that the origin is in an infinite connected set of
eventually occupied vertices. Consider polluted bootstrap percolation on $\Z^d$ with $d\ge 3$, threshold
  $\thr=2$, density $p>0$ of initially occupied vertices,
  and density $q>0$ of closed vertices.  Theorem $1$ of~\cite{GH} states that,
  for both the standard and modified models,
$$
\label{GH-thm1} \P\bigl(\text{\rm the origin
is eventually occupied}\bigr)\to 1\qquad\text{as }(p,q)\to (0,0),
$$
and moreover, $\P(\perc)$
also tends to $1$.

In this article we treat polluted bootstrap percolation on
$\Z^3$ with threshold $\thr = 3$. Our strongest result is
for the modified model given by (\ref{modified}). Similarly
to the case $d = \thr = 2$ of [GM2], but in contrast with
the $d > \thr = 2$ case of \cite{GH}, the final density
here depends on the $p$ versus $q$ scaling, but now with a
cube law (modulo logarithmic factors).

\begin{thm}\label{three-modified}
  Consider modified polluted bootstrap percolation (rule~(\ref{modified})) on
  $\Z^3$ with threshold
  $\thr=3$, density $p$ of initially occupied vertices, and density $q$ of
  closed vertices.  There exists a constant $C\in(0,\infty)$ for which the
  following hold.
\begin{enumerate}[label= \textup{(\roman*)}]
\item \label{three-modified-i} If $p,q\to 0$ in such a way that $q=o(p^3(\log p^{-1})^{-3})$ then the probability
    that the origin is eventually occupied tends to $1$,
    and indeed so does $\P(\perc)$.
\item \label{three-modified-ii} If $p,q\to 0$ in such a way that $q>Cp^3$ then the probability
    that the origin is eventually occupied tends to $0$,
    and moreover $\P(\perc)=0$ for small enough $p$.
\end{enumerate}
\end{thm}

Our methods rely on the technology of
oriented surfaces introduced recently in \cite{DDGHS}.
The article \cite{GH} also used this technology,
but for different purposes.  The proof in \cite{GH} involves the
construction of an oriented open surface on which occupation is able to spread.
Our proof of Theorem~\ref{three-modified} (i) relies crucially on the main result of \cite{GH} for the threshold $r=2$ model (as stated above), together with a relatively straightforward renormalization argument.  Interestingly, we need almost the full power of the result of \cite{GH}, since the renormalized system that we apply it to has $p,q\to 0$ with $p$ much smaller than $q$.  We do not know any other route to proving this bound. The main contribution of this article is Theorem~\ref{three-modified} (ii), whose proof will use elaborate oriented surfaces to block growth rather than to facilitate it.

Turning to the standard model, when $q>Cp^3$, we are able to construct
infinite blocking surfaces
that prevent the spread of occupied vertices from one direction.
However, our approach requires
a number of surfaces of different orientations to
be stitched together to create a finite blocking surface
of diameter a power of $p^{-1}$; this size restriction is
needed to prevent any substantial
growth of occupation on the inside.
For the stitching not to leak occupation,
we need a larger density of occupied vertices, resulting in the
following weaker result in which the powers in the two bounds do not match each other.
It is a very interesting open problem to determine the right $p$ versus $q$ scaling for the standard model.

\begin{thm}\label{three-standard}
  Consider standard bootstrap percolation (rule~(\ref{standard})) on $\Z^3$
  with threshold
  $\thr=3$, density $p$ of initially occupied vertices, and density $q$ of
  closed vertices.  There exists a constant $C\in(0,\infty)$
  for which the
  following hold.
\begin{enumerate}[label= \textup{(\roman*)}]
\item \label{three-standard-i} If $p,q\to 0$ in such a way that
$q=o(p^3(\log p^{-1})^{-3})$, then the probability
    that the origin is eventually occupied tends to $1$,
    and indeed so does $\P(\perc)$.
\item \label{three-standard-ii} If $p,q\to 0$ in such a way that $q>Cp^2$, then the probability
    that the origin is eventually occupied tends to $0$,
    and moreover $\P(\perc)=0$ for small enough $p$.
\end{enumerate}
\end{thm}

As mentioned in \cite{Mor} and elaborated in \cite{GH}, it is easy to show directly that
when $d=\thr\ge 2$,
the final density tends to $0$ if $q$ exceeds
some small power of $p$.  In addition to addressing the case $d=r=3$, the
(ii) parts of Theorems~\ref{three-modified}
and~\ref{three-standard} yield some improvements to the powers for $d=r\geq 4$.
Indeed, assuming everything outside
$\{0\}^{d-3}\times \Z^{3}$ (respectively $\{0,1\}^{d-3}\times \Z^{3}$)
is occupied,
we can deduce that when $d=\thr\ge 3$, the final density
goes to $0$ if $q\gg p^3$ for the modified
model (respectively if  $q\gg p^{2^{4-d}}$ for the standard model).
No power-law inequality between $p$ and $q$ guaranteeing that the final
density goes to $1$ is currently known for either model when
$d=\thr\ge 4$. (See the problems (i) and (ii) in Section 6 of
\cite{GH} for open questions for $d\ge 4$.)

If obstacles are made slightly larger, then we can obtain
matching upper and lower bounds (up to logarithms) for the
standard model also. More precisely, let the initial
configuration be chosen as follows. Independently mark each
vertex as an \df{obstacle center} with probability $q$.
For each obstacle center, declare that vertex and each of
its $6$ neighbors closed; all other vertices of $\Z^3$ are
declared open. Then, conditional on the set of open
vertices, declare each open vertex independently to be
initially occupied with probability $p$.  Call this the
\df{big obstacles} initial configuration.

\begin{thm}\label{larger obstacles}
  Consider standard or modified bootstrap percolation (rule \eqref{standard} or \eqref{modified}) on
  $\Z^3$  with threshold
  $\thr=3$, and the big obstacles initial configuration with density $p$ of reserved vertices and density $q$ of
  obstacle centers.  There exists a constant $C\in(0,\infty)$ for which the
  following hold.
\begin{enumerate}[label= \textup{(\roman*)}]
\item \label{large obstacles-i} If $p,q\to 0$ in such a way that $q=o(p^3(\log p^{-1})^{-3})$ then the probability
    that the origin is eventually occupied tends to $1$,
    and so does $\P(\perc)$.
\item \label{large obstacles-ii} If $p,q\to 0$ in such a way that $q>Cp^3$ then the probability
    that the origin is eventually occupied tends to $0$,
    and $\P(\perc)=0$ for small enough $p$.
\end{enumerate}
\end{thm}


\subsection*{Notation}
Two norms will be used throughout the paper: the $\ell^\infty$ norm is denoted $\|\cdot\|_\infty$ and the $\ell^1$ norm is denoted $|\cdot|$. When describing subsets of $\Z^3$, intervals denote their intersections with the integers, so for real numbers $a<b$ we write $[a,b) := [a,b)\cap \Z$, etc. In a deviation from commonly used conventions, it is useful for us to define $(b,a] = [a,b)$, and similarly for other intervals.  We use both ``vertex'' and ``site'' for elements of $\Z^3$, but in different contexts.
Vertices refer to points in the original lattice, which can be occupied, closed or empty, while sites refer to the locations of rescaled boxes, as identified by points in $\Z^3$.  Points that will \emph{eventually} refer to the locations of rescaled boxes at some later time in the proof are also referred to as sites, as in \cref{defn of shell,existence of shell}.

\section{Comparison result and outline of proofs}\label{outline}
\cref{three-modified} has two parts.  The lower bound states that the origin is eventually occupied with high probability if $q$ is small compared with $p^3$.  As mentioned earlier, this is derived via a relatively straightforward renormalization argument from the threshold $\thr=2$ result of the companion paper \cite{GH}.

The main contribution of this paper is the upper bound, which states that the origin remains unoccupied with high probability if $q$ is large compared with $p^3$.  At the heart of the proof is the following simple but subtle deterministic result comparing the
$\thr=3$ and $\thr=2$ models on a suitable set of vertices, with different boundary conditions. To state the result for the modified model, for a set $Z$ and $x\in Z$, we define
$$\nonbrs(x)=\nonbrs_Z(x):=\#\bigl\{i=1,2,3: x-e_i\notin Z \text{ or }x+e_i\notin Z\bigr\}$$
to be the number of coordinates in which $x$ has a neighbor outside $Z$.  For use in the context of  the standard model, we also let $\nonbrs'(x)=\nonbrs'_Z(x)$ be the total number of neighbors of $x$ outside $Z$.

\begin{prop} \label{super-sneaky}
Fix an integer $m\ge 1$. Fix a finite set $Z\subset \Z^3$, and
run two modified bootstrap percolation dynamics: the first with threshold $\thr=3$ and
$Z^c$ initially occupied; the second with threshold
$\thr=2$ and $Z^c$ closed.
Assume that the configuration on $Z$ satisfies the following
conditions.
\begin{enumerate}[label= \textup{(\roman*)}]
\item  Any $x\in Z$ with $\nonbrs(x)= 3$ is a closed vertex.
\item  For any $x\in Z$ with $\nonbrs(x)\ge 2$, there is no
initially occupied vertex within $\ell^\infty$ distance $m$ of $x$.
\item The final configuration in the second dynamics has no connected set
of occupied vertices with $\ell^\infty$-diameter larger than $m/2$.
\end{enumerate}
Then any vertex $x\in Z$ that is occupied at any
time by the first dynamics is also occupied by that time
in the second dynamics.

For standard bootstrap percolation,
the same statement holds  with
$\nonbrs$ replaced by $\nonbrs'$.
\end{prop}
\begin{proof} Consider the modified rule;
the proof for the standard rule is nearly identical.
Assume the conclusion does not hold, and consider the
first time $t$ at which a vertex $x\in Z$ is occupied by the
first but not by the
second dynamics. As the two dynamics have the same
initial configuration on $Z$, we have $t>0$.
Then we cannot have $\nonbrs(x)\ge 3$, since closed vertices do not change.
We cannot have $\nonbrs(x)=2$ either, as $x$ has no occupied neighbors
in $Z$ at time $t-1$ by minimality of $t$. Thus
$\nonbrs(x)\le 1$, but then $x$ has an
occupied neighbor in at most one coordinate
outside of $Z$ in the first dynamics at time $t-1$, and therefore must also get occupied by the
second dynamics, a contradiction.
\end{proof}

We will apply \cref{super-sneaky} by carefully constructing a suitable random set $Z$ containing the origin.  This set will have diameter at most $p^{-s}$ for some large but fixed constant $s$.  For small $p$, this is much smaller than the critical size $e^{c/p}$ for threshold
$\thr=2$ bootstrap percolation on finite sets.  Consequently, it will follow from standard bootstrap methods (e.g.\ of \cite{AL}) that with high probability the $\thr=2$ model restricted to $Z$ does not occupy the origin and does not produce large occupied clusters (as appearing in condition (iii) of \cref{super-sneaky}), even if we reassign all internal closed vertices of $Z$ to be open. The construction of the set $Z$ will be an involved and delicate task.  We therefore explain some of the ideas before starting on the technical details.  \cref{Zfig} illustrates the key features of the (random) set $Z$ given by our construction. 

\begin{figure}
\begin{center}
\includegraphics[width=.95\textwidth]{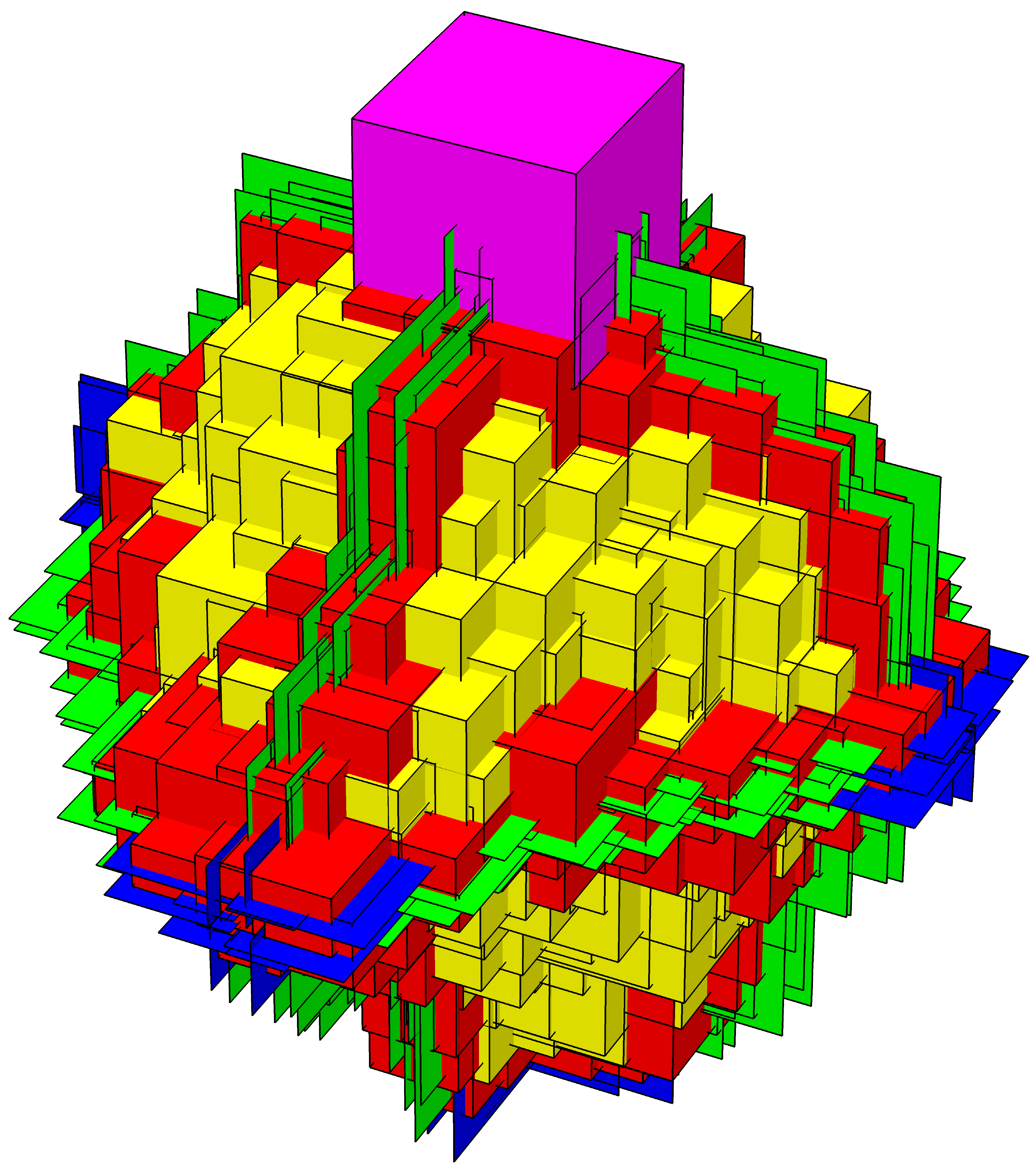}
\end{center}
\caption{\label{Zfig} The ``stegosaurus,'' $Z$, with only one of the six keystones shown, in magenta. Cuboids are shown in red or yellow, according to whether they are protected by a plate or not. Plates are shown in blue or green, according to whether they need protection by a keystone or not.  Black lines indicate ``exposed'' edges of cuboids and plates, which must be near nice vertices at the corners, and be protected by other cuboids or plates nearby.}
\end{figure}

To see why it is reasonable to expect such a set to exist,
consider first the simpler problem of protecting from
occupation from a single direction, say $(1,1,1)$.
Specifically, suppose that all vertices in the half space
$\{x: x_1+x_2+x_3> 0\}$ are initially occupied.  In the
absence of closed vertices, the occupied set will advance
deterministically to $\{x: x_1+x_2+x_3> -1\}$ at the next
step, and so on, so that all of $\Z^3$ is eventually
occupied.  Now suppose instead that the origin (say) is
closed, and no vertices in the negative octant
$(-\infty,0]^3$ are intially occupied.  Then this octant is
protected from the advancing occupation and remains empty
forever.  On the other hand, if the origin is closed but
some vertex $(-\ell,0,0)$ on a negative coordinate axis is
initially occupied, then the axis
$(-\infty,0)\times\{0\}^2$ will become fully occupied.  If
all three negative axes similarly contain an initially
occupied vertex, then once again all (open) vertices will
become occupied.  In a random configuration, we can expect
the negative axes to be free from initially occupied
vertices up to a distance $L=c/p$ with reasonable
probability (where~$c$ is small constant), so that the
octant is temporarily protected until the occupied half
space advances by $L$. Moreover, if $L^3 q$ is large, then
we can expect to find some closed vertex (the origin in
this example) in a $3$-dimensional region of length scale
$L$, so such a temporary protection is common.  This
computation is the basic reason behind the $q$ versus $p^3$
scaling. It is crucial that we need only forbid initially
occupied vertices on the $1$-dimensional edges of the
region being protected, not its $2$-dimensional faces
(which would give a different scaling).

Continuing with the example above, to make the temporary
protection permanent we need to find and use further closed
vertices before we encounter initially occupied vertices
along the axes.  Because of the $q$ versus $p$ scaling, we
need to look for these closed vertices not on the axes but
in $3$-dimensional regions.  Suppose that the origin is
closed and there are no initially occupied vertices in the
interval $[-L,0)\times\{0\}^2$ of the axis, and that in
addition the vertex $(-L,L,L)$ is closed.  This vertex
generates its own octant $(-L,L,L)+(-\infty,0]^3$, and the
axis $(-\infty,0)\times\{0\}^2$ pierces one of its faces,
obviating the need for the rest of the axis
$(-\infty,-L)\times\{0\}^2$ to be free of initially
occupied vertices.  More generally, suppose that we have an
infinite set of closed vertices.  Each generates an octant,
and suppose that each of its three edges pierces a face of
another octant before encountering an initially occupied
vertex.  Then the octants will protect each other, and all of
them will remain unoccupied forever.

For such an arrangement to exist in the random setting, the
infinite set of closed vertices discussed above should be
spaced at length scale about $L$, and should be arranged in
a kind of oriented surface, which can be regarded as a
random perturbation of the hyperplane $\{x:x_1+x_2+x_3=0\}$
with rather strict conditions on its local geometry.  We
will construct such a surface by considering renormalized
boxes of scale $L$ and by adapting the recent duality
technology introduced in \cite{DDGHS} (and developed in
\cite{GH1,GH2,GH3}) for constructing oriented surfaces in
percolation models.

We now return to the harder problem of constructing a
finite set $Z$ to protect the origin from occupation from
\emph{all} directions.  We can imagine that every vertex
outside some very large $\ell^1$-ball is initially
occupied, and we want to conclude that the origin remains
unoccupied.  The rough idea is to surround the origin by an
envelope of closed vertices at spacing about $L$, each of
which protects the cuboid having opposite corners at the
closed vertex itself and at the origin -- these cuboids are
colored red and yellow in \cref{Zfig}. There should be no
initially occupied vertices on the ``exposed'' edges (black
lines in \cref{Zfig}) of these cuboids before they pierce
others, which should happen within distance about $L$.  The
set $Z$ will be the union of the cuboids.

One approach to constructing an envelope as described is to
combine eight oriented surfaces of the previous type in
various directions, with normal vectors $(\pm 1,\pm 1,\pm
1)$, to enclose the origin in an envelope in the shape of a
perturbed regular octahedron ($\ell^1$-sphere).  However,
as we will discuss below,  complications arise at the edges
and corners of the octahedron, where two or more surfaces
intersect. It turns out to be easier to control the
geometry of edges and corners if, instead of intersecting
surfaces, we adapt the percolation duality methods
of~\cite{DDGHS} to construct the envelope directly. (A
similar method appeared in \cite{GH1}.)  The resulting
shape is still a perturbed octahedron, but its edges are
guaranteed to lie on the coordinate planes (at the level of
renormalized sites of scale $L$).

The edges and corners of the octahedron require special
treatment, essentially because they are vulnerable to
occupation from more directions.  For simplicity, suppose
that the set $Z$ approximates the octahedron $\{x: |x|\leq
t\}$ and that it includes a closed vertex at
$x=(t/2,t/2,0)$, which is the center of an edge of the
octahedron.  The line $\{(t/2,t/2)\}\times\Z$ will contain
initially occupied vertices on both sides of $x$ (typically
at distance of order $1/p$, as usual).  But we should not
expect the envelope to include any closed vertex with the
first two coordinates both greater than $t/2$, since it
would have $\ell_1$-norm greater than $t$.  Therefore,
there is nothing to protect this line, and it is vulnerable
to becoming fully occupied (except at $x$).  The conclusion
is that the vertex $x$ cannot itself protect a
3-dimensional cuboid, but only the 2-dimensional plate
$[0,t/2]^2\times\{0\}$.

An important difference between the modified and standard
models appears here.  In the standard model, even the plate
$[0,t/2]^2\times\{0\}$ mentioned above is not safe from
occupation from outside.  Vertices in its interior have two
potentially occupied neighbors on either side of the plate,
so one initially occupied vertex in the plate will cause the
entire plate to become occupied.  To prevent this, the
plate must be thickened to thickness at least $2$, and must
have closed vertices at both its outermost corners (perhaps
at $(t/2,t/2,\pm L)$, for instance).  But this means that
we need two closed vertices on the same axis-parallel line,
with no initially occupied vertex between them.  That
requires a different $q$ versus $p$ scaling, and is the
reason that our upper and lower bounds for the standard
model do not match. (However, if obstacles are made larger
as in \cref{larger obstacles}, then these plates have
thickness at least $2$ at no additional cost, and we get
the same $q$ versus $p$ scaling for the standard and
modified models.)

Returning to the modified model, our above assumption that
there was a closed vertex $x$ exactly on the
coordinate plane $\Z^2\times\{0\}$ was in fact an
unrealistic oversimplification.  Since the renormalization
scale $L$ is chosen so that $L^3 q$ is large, finding a
closed vertex typically requires a region of volume $L^3$.
So a more realistic choice is $x=(t/2,t/2,0)+z$ for some
(random) $z\in(-L,L)^3$.  In this case, the closed vertex
$x$ protects a plate $[0,x_1]\times[0,x_2]\times\{x_3\}$
that does \emph{not} include the origin.  However, it can
still protect the cuboids generated by nearby closed
vertices on the faces of the shell.  This necessitates a
further complication: these cuboids should be modified so
as to extend exactly up to the plate, rather than to the
coordinate plane.  The plate itself can be protected by
nearby cuboids. These plates are colored green and blue in
\cref{Zfig}.

Our set $Z$ resembles a stegosaurus.  The overall shape is
a perturbed octahedron ($\ell^1$-ball), with a rough
surface composed primarily of corner regions of randomly
placed cuboids.  Along each edge of the octahedron, there
is a row of protruding plates parallel to the edge, to
protect the vulnerable spine.  Each cuboid and each plate
has a closed vertex at its outermost tip, and no initially
occupied vertices on its exposed edges.  The set $Z$ is the
union of all cuboids and plates.  The locations of plates
vary in the direction perpendicular to themselves
--- they do not all lie in the same plane.  In fact, it is
useful to have \emph{two} rows of plates side by side, on
each side of the coordinate plane, and for the plates to
protrude slightly farther than the basic octahedron shape
would suggest, further enhancing the stegosaurus
comparison.  A cuboid close to the coordinate plane is
protected by plates on the far side of the plane (red
cuboids in \cref{Zfig}), while protecting plates on the
near side (green plates in \cref{Zfig}).  The extra
protrusion ensures that plates protect nearby cuboids
despite random fluctuations in the shape of the shell.

We have not yet considered the corners of the octahedron.
Here there is a serious issue.  Like the head and tail of a
stegosaurus, the corners are especially vulnerable to
attack, and require extra protection.  The problem arises
for a closed vertex on the surface of $Z$ close to the coordinate
axis, such as the closed vertex $y$ with the largest
positive first coordinate, which will be close to
$(t,0,0)$.  By similar considerations to those concerning
the edges, this closed vertex can only protect the
$1$-dimensional ray $[0,y_1]\times\{(y_2,y_3)\}$, and is
thus essentially useless for protecting other nearby plates
and cuboids. Therefore, the vertex with the \emph{second}
largest first coordinate will have a similar issue, and so
on, unravelling the entire scheme!

Our solution is rather extravagant.  Suppose that the cube
of side length $20L$ (say) centered at $(t,0,0)$ has closed
vertices exactly at all $8$ corners, and no initially
occupied vertices on its edges.  It is easy to conclude
that this cube can never be invaded by occupation from
outside, and therefore it acts as a keystone, protecting
all plates and cuboids nearby, and stabilizing the entire
structure.  This of course comes at a cost.  The
probability of the above event is very small, of order
$q^8$, and since we will need a keystone at each of the $6$
corners of the octahedron, the probability becomes
$q^{48}$.  One keystone (of the $6$) is shown in magenta in
\cref{Zfig}.   (Regardless of the details of the
construction, it appears that this probability must be
$o(1)$, since any variant of the keystone construction must
involve two closed vertices on the same axis parallel
line).   But the key point is that this probability is a
constant power of $q$ (equivalently, of $p$).

We can make
many attempts to find a shell enclosing the origin, each
larger than the previous one.  At each attempt, the random
surface construction succeeds with at least probability
$1/2$, say, regardless of the size of the surface, because
it is based on percolation arguments.  On the other hand,
the keystones only exist with probability $q^{48}$, so we
need to make about $q^{-48}$ attempts before we succeed.
(Or rather $q^{-49}$, say, to succeed with high
probability).  The resulting set will be very large, but,
as promised earlier, its size will be at most polynomial in
$1/p$.

One more complication was glossed over so far.  A surface
of the kind described above can protect the origin from
occupation from outside, but there will also be initially
occupied vertices inside it (i.e.\ in $Z$), including on or
near the faces of the cuboids.  The internal dynamics might
interact with the external dynamics through the faces,
causing a vertex on an edge of a cuboid to become occupied,
again leading to disaster.  This is where the comparison
with internal threshold $\thr=2$ dynamics in
\cref{super-sneaky} is needed.  The polynomial size of $Z$
will ensure that internal clusters have diameter bounded by
some fixed (but large) constant $m/2$ with high
probability.  Therefore, for each closed vertex $x$ that
makes up our surface, we will require absence of initially
occupied vertices not only on the surrounding axis-parallel
line segments of length of order $L$, but also within
$\ell^\infty$ distance $m$ of these line segments in all
directions.  This requirement must of course be taken into
account in the percolation and renormalization calculations
that allowed the surface to be constructed.  This creates a
somewhat delicate interplay between the various constants,
but it turns out that they can all be chosen appropriately,
as summarized below.

Turning to some further details, we will construct the set
$Z$ via renormalization.  A vertex $u$ will be declared
``nice'' if it is closed and there are no initially
occupied vertices within distance $m$ of any axis-parallel
line-segment from $u$ of length some multiple of $L$.  All
external corners of $Z$ will be nice vertices.  The
parameters will be chosen so that a cube of side $L$
contains a nice vertex with high probability.  In fact, it
will be convenient to control the approximate placement of
nice vertices within the cube, so that they can be chosen
on the outermost sides of $Z$, allowing for the protrusion
and the double row of plates discussed above. Therefore, we
will consider cubes twice the size, of side $2L+1$.  We will call such a
cube good if all eight of its side-$L$ subcubes contain a
nice vertex.  We will find a ``shell'' of good cubes
containing the origin. Following the approach of
\cite{DDGHS,GH1}, the shell will be constructed via
duality, as the boundary of a set reachable via paths of
boxes of a certain carefully chosen type. This will allow us to
accurately control its geometry.  The set $Z$ will be
constructed using the nice vertices in these cubes, with
the geometric constraints ensuring that the various
mutual protection conditions hold, provided keystones are present.

\subsection*{Symmetry conventions}
Our construction of the random set $Z$ satisfying \cref{super-sneaky} will be symmetric (in distribution) under permutations of coordinates and reflection through coordinate planes (sign-flips).  Therefore, we will state and prove many of the preliminary lemmas in \cref{existence of shell,Z is super-sneaky} for vertices in the positive octant and on the positive coordinate planes --- analogous statements clearly hold by symmetry for vertices in the other octants and coordinate planes, but we omit these statements for the sake of readability.

\subsection*{Choice of key constants}

A large integer $s$ is chosen so that a box of
diameter $p^{-s}$ is likely to contain
a successful ``stegosaurus'' $Z$. This number only
depends on the probability of the occurrence of $6$ keystones
at the specific locations, and its value is determined
(to be $s=300$) in the proof of Lemma~\ref{Z exists}.

The large integer $m$ is the radius of the initially unoccupied regions
around the edges  of $Z$ in
Proposition~\ref{super-sneaky}. Ultimately, $m$ depends
on the size of $Z$, and therefore on $s$,
as it depends on how much the threshold $r=2$
dynamics are likely to achieve inside $Z$. As a consequence of
Proposition~\ref{threshold 2}, the dependence is
a simple linear one ($m=12s$), and leads to the
choice of $m$ also in the proof of Lemma~\ref{Z exists}.

We will need a small parameter $\delta>0$, which determines the length scale
$L=\floor{\delta/(m^2p)}$. The choice of $\delta$ and the constant $C>0$ from the statement of \cref{three-modified} determine the probability that a rescaled site (a box of diameter $2L+1$) has enough strategically placed closed vertices and initially unoccupied vertices (see \cref{sec:good boxes}).  \cref{good-likely} implies this probability is at least $1-\epsilon$ when we choose
$\delta = \epsilon/(16\cdot 10^5)$ and $C$ is chosen sufficiently large depending on $m$ and $\epsilon$ (from the proof of \cref{good-likely}, we can take $C = (16\cdot 10^5 m^2/\epsilon)^3 \log(16/\epsilon)$).
To deal with finite-range
dependence between rescaled sites, we use \cite{LSS} to determine $\epsilon>0$ in the proof of \cref{success prob}.
Thus, we do not give an explicit
value to $\epsilon$, so neither $\delta$ nor $C$ are given explicit values.

We also emphasize that, after the values of the constants mentioned above are determined, $p$ needs to be assumed small enough
(depending on all these values) for
all the arguments to work properly.
We also need to assume that $q$ is
small enough (see Sections~\ref{sec:good boxes},
\ref{all together}, and~\ref{standard model}), which we may, since the probability that the origin is eventually occupied and $\P(\perc)$ are decreasing in $q$.

\section{The lower bound}\label{low-q}



In this section we prove
\cref{three-modified}(i), which also immediately implies \cref{three-standard}(i). Thus, we consider
the modified model for the rest of this section.
Pick an integer $N\ge 1$. For now, $N$ is arbitrary, but later we
choose it to be on the order a bit larger that $p^{-1}$.
A site $x\in \Z^3$ is called \df{$N$-open}
if the box $Nx+[0,N)^3$
contains no closed vertices and every nonempty intersection between a line parallel to a coordinate axis and $Nx+[0,N)^3$
contains an initially occupied vertex;  $x$ is called \df{$N$-closed} otherwise. Moreover,
a site $x$ is called \df{$N$-occupied} at some time $t$ if  the box $Nx+[0,N)^3$ is
fully occupied at that time.

\begin{lemma}\label{theta3m-growth} Choose any $N\ge 1$.
If $x\in \Z^3$ is $N$-open, and it has two nearest neighbors
$y_1$ and $y_2$ with $||y_1-y_2||_\infty=1$ (i.e.,
$y_1$ and $y_2$ are neighbors in two
different coordinates), which are both
$N$-occupied at some time $t$, then $x$ is $N$-occupied at some
later time.
\end{lemma}

\begin{proof}
This is an easy verification.
\end{proof}

\begin{lemma}\label{theta3-Lopen}  Let
$N=\lfloor 3p^{-1}\log p^{-1}\rfloor$ and assume
$q=o(p^3(\log p^{-1})^{-3})$. Then
the probability that $0$ is $N$-open converges to
$1$ as $p\to 0$.
\end{lemma}

\begin{proof}
The probability that there exists a nonempty intersection between a line parallel to a coordinate axis and $[0,N)^3$
that fails to contain an initially occupied vertex
is bounded above  by
$
3N^2(1-p)^N\le 3N^2\exp(-3\log p^{-1})\le 30p(\log p^{-1})^2.
$
Moreover, the probability that this box contains
a closed vertex is at most $qN^3$, which by
the assumption approaches $0$ as $p\to 0$.
\end{proof}

\begin{proof}[Proof of Theorem~\ref{three-modified}(i) and
Theorem~\ref{three-standard}(i)]
Assume that $N$ is as in Lemma~\ref{theta3-Lopen}.
Run the threshold $r=2$ modified bootstrap rule on initially $N$-occupied and $N$-closed sites in $\Z^3$. That is, initialize a bootstrap percolation with a configuration of closed and occupied vertices corresponding to those sites that are initially $N$-closed and $N$-occupied in the rescaled configuration.
Let $Y$ be the resulting set of eventually occupied
vertices in this threshold $2$ process; \cref{theta3m-growth} guarantees that the set of eventually occupied vertices in the original, threshold $3$ process contains the set $NY + [0,N)^3$.
The probability that a site is initially $N$-closed converges to $0$ as $p\to 0$ by \cref{theta3-Lopen}, the probability that a site is initially $N$-occupied is $p^{N^3}>0$, and different sites are $N$-closed and $N$-occupied independently.
Theorem 1 of~\cite{GH} (restated here on page~\pageref{GH-thm1}) thus guarantees that, with probability approaching $1$ as $p\to 0$, the set
$Y$ contains an infinite connected set that includes the origin,
and therefore so does $NY+[0,N)^3$.
%
\end{proof}

\section{Definition of the Shell}\label{defn of shell}
In this section, we define a shell to be a subset of $\Z^3$
having certain properties. Later, this shell will consist
of rescaled boxes having certain good properties as defined
in \cref{sec:good boxes}, which will be used to construct
the stegosaurus.

Let $a\in \Z^3$ have no coordinate equal to zero.
We say that a site $x\in \Z^3$ is \df{$a$-protected by $y\in \Z^3$} if
\begin{itemize}
\item$y-x \in (0, a_1]\times (0,a_2]\times (0,a_3]$, or
\item $y = (0, y_2, y_3)$ and $y-x \in (0, a_1]\times [0,a_2]\times [0,a_3]$, or
\item $y = (y_1, 0, y_3)$ and $y-x \in [0, a_1]\times (0,a_2]\times [0,a_3]$, or
\item $y = (y_1, y_2, 0)$ and $y-x \in [0, a_1]\times [0,a_2]\times (0,a_3]$.
\end{itemize}
In other words, if one of the coordinates of $y$ is zero, then the intervals in the other two coordinates are allowed to include $0$.

Let $E\subseteq \Z^3$. A site $x\in [1, \infty)^3$ is called \df{protected by $E$} if for each
\begin{equation}\label{eq:protected1}
a\in \{(-3, 3, 3), (3, -3, 3), (3, 3, -3)\}
\end{equation}
 there exists a corresponding $y\in E$ such that $x$
 is $a$-protected by $y$.
 A site $x \in [1,\infty)\times [1, \infty)\times\{0\}$ is called \df{protected by $E$} if for each
\begin{equation}\label{eq:protected2}
 a\in \{(-3, 3,  3), (-3, 3,  -3), (3, -3,  3), (3, -3, -3)\}
 \end{equation}
 there exists a corresponding $y\in E$ such that $x$ is $a$-protected by $y$. Similarly, for $x\in \Z^3$ with three or two non-zero coordinates, we say $x$ is \df{protected by $E$} if it satisfies the definition above with the signs of the coordinates flipped in displays~\eqref{eq:protected1} or~\eqref{eq:protected2} in an identical way throughout. For example, if $x \in (-\infty, -1]\times [1,\infty)^2$, then we flip all the first coordinates, and replace~\eqref{eq:protected1} with ``$a\in \{(3, 3, 3), (-3, -3, 3), (-3, 3, -3)\}$''. For $x$ with zero or one non-zero coordinates, we will not need to refer to $x$ as being protected.

\begin{figure}
\begin{center}
\includegraphics[width=.27\textwidth]{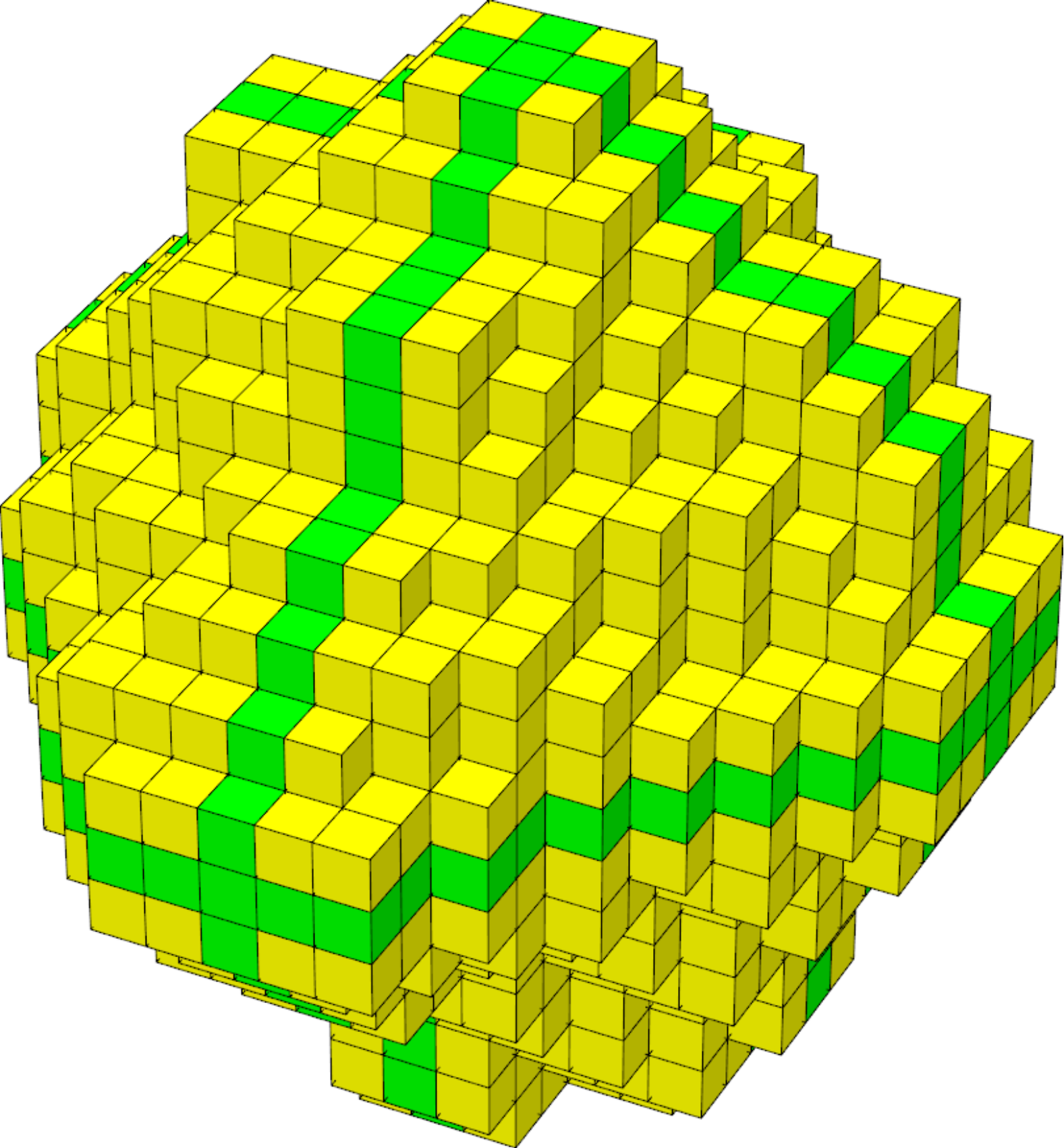} \hspace{.05\textwidth}
\includegraphics[width=.27\textwidth]{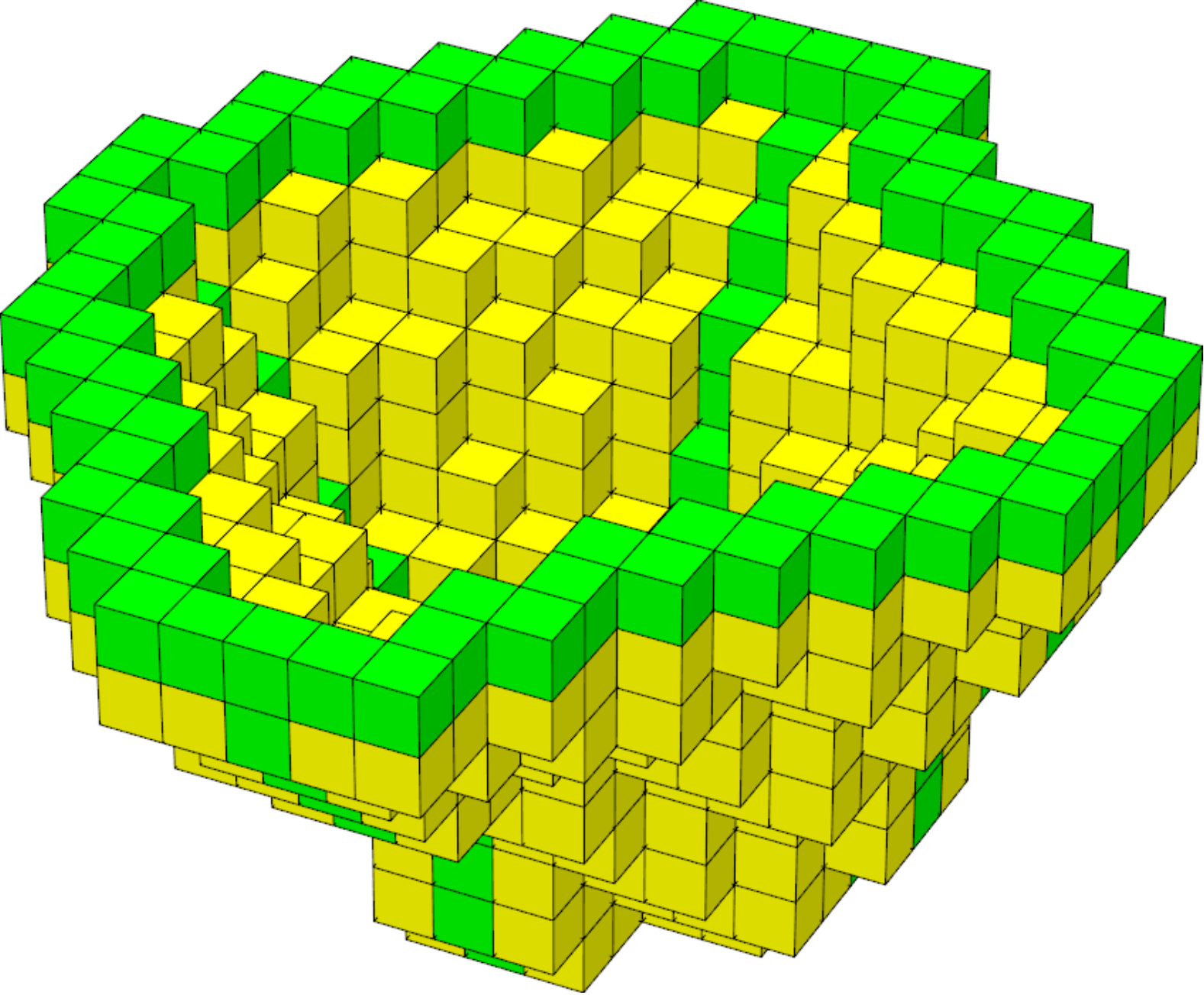}\\[1cm]
\includegraphics[width=.27\textwidth]{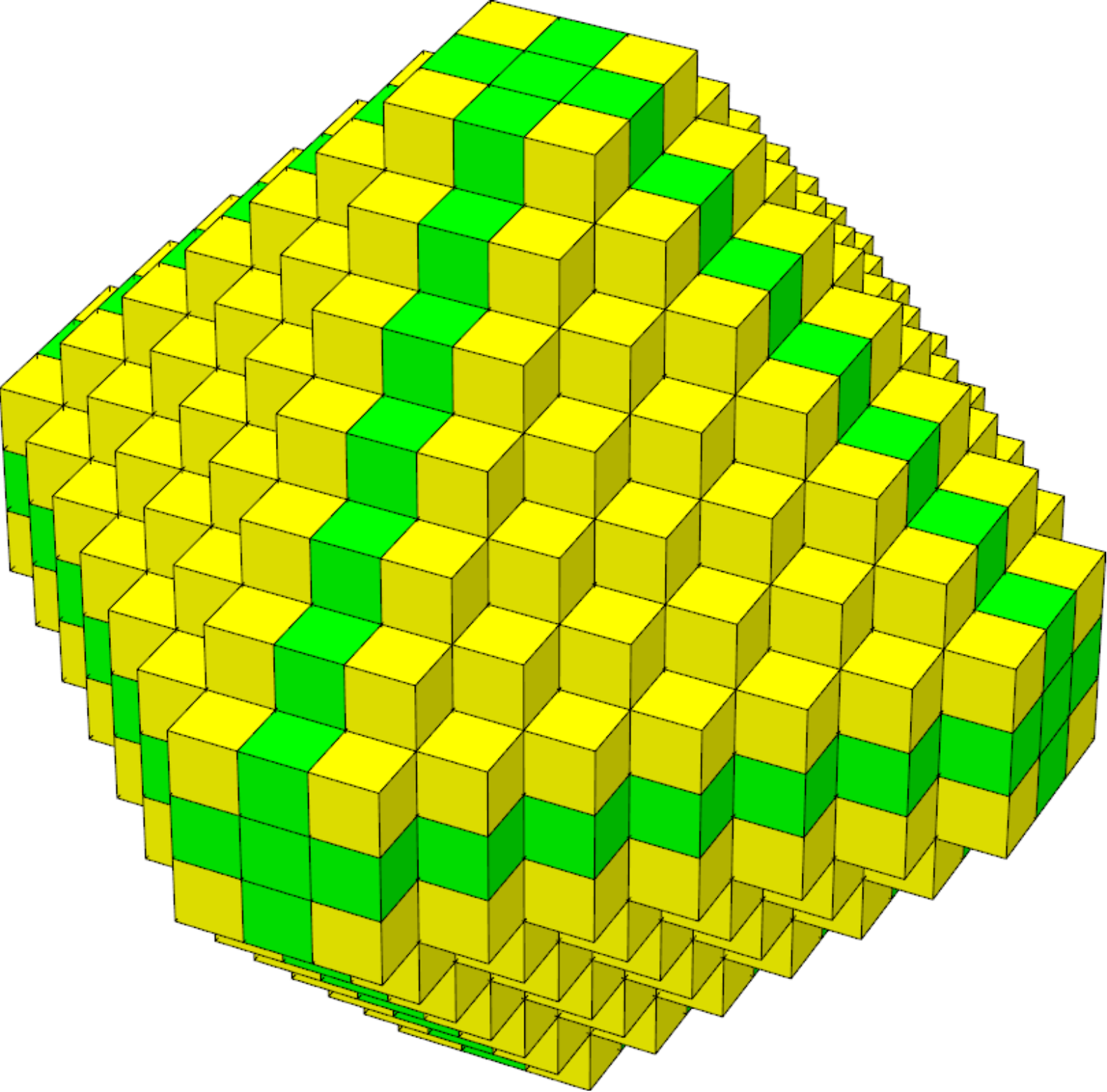}  \hspace{.05\textwidth}
\includegraphics[width=.27\textwidth]{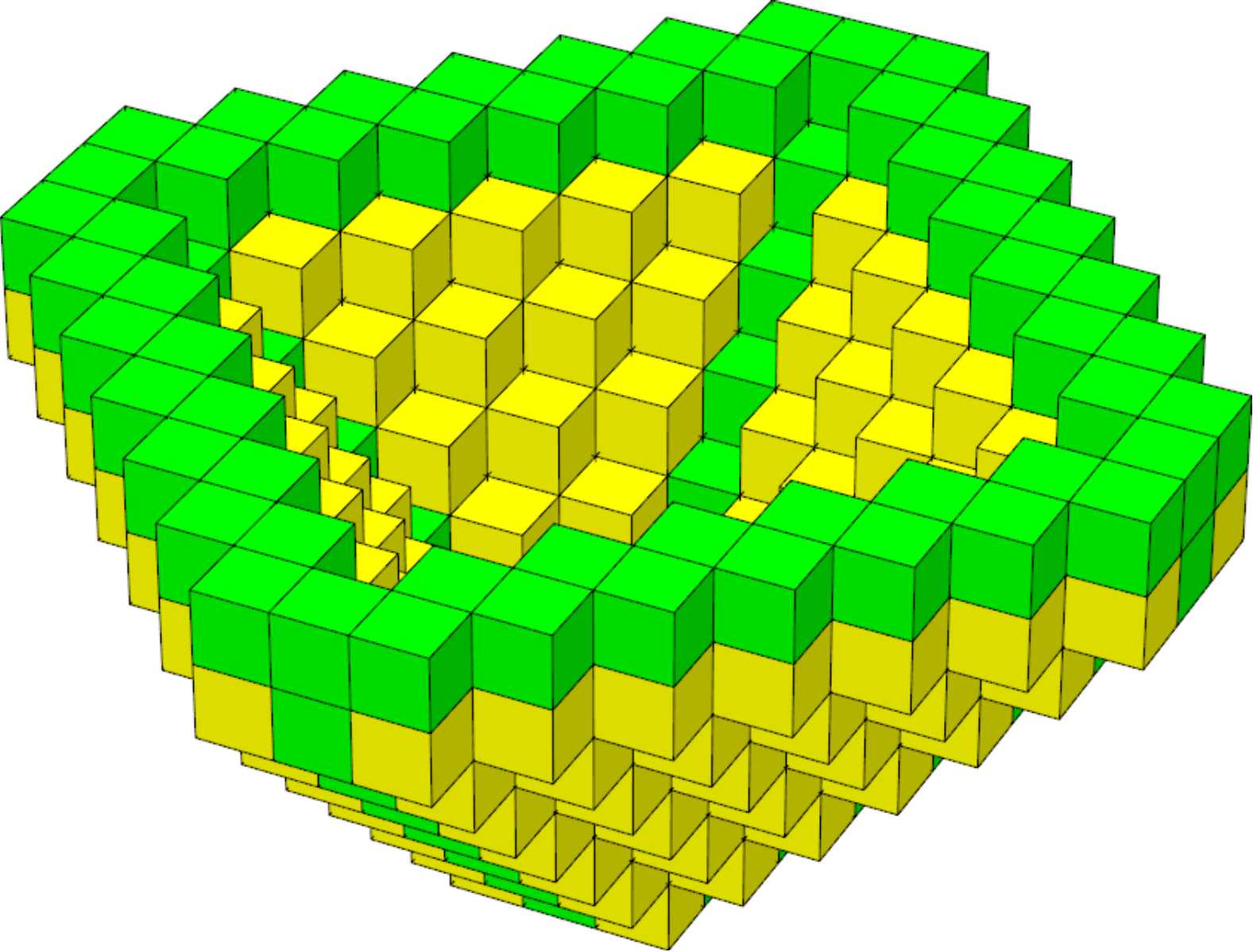}
\end{center}
\caption{{\em Top:} a set satisfying properties (S2)--(S4) of a shell; its bottom half is also shown, so that the interior can be seen.  The spine is in green.  {\em Bottom}: the additional condition (S1) is satisfied provided the shell is deterministically ``flat'' near its six corners, as here.  (A non-trivial example satisfying (S1) would be too large for a convenient illustration.)}
\label{fig-shell}
\end{figure}
A \df{shell $S$ of radius $n$} is defined to be a non-empty subset of $\Z^3$ that satisfies the following properties.  See \cref{fig-shell}.
\begin{enumerate}[label=(S\arabic*)]
\item \label{S1} The shell $S$ contains all sites $x\in
    \Z^3$ such that $|x|=n$ and $\|x\|_\infty \ge n-12$.
    (This implies that $S$ contains portions of the
    $|\cdot|$-sphere of radius $n$ in neighborhoods of
    each of the six sites $(\pm n, 0, 0), (0,\pm n, 0)$
    and $(0, 0, \pm n)$.)
\item \label{S2} For each $x\in S$, we have $n\le |x| \le n+3\sqrt{n}$
and
$\|x\|_\infty\le n$.

\item \label{S3} For each $x\in S$ that has at most one coordinate that is less than $4$ in absolute value, $x$ is protected by $S$.

\item \label{S4} For each of the eight directions $\varphi \in \{(\pm1, \pm1, \pm1)\}$, there exists an integer $k = k(\varphi)\ge n/3$ such that $k\varphi \in S$.
\end{enumerate}

If $S$ is a shell, the intersection of $S$ with the union of the three coordinate planes is called the \df{spine} of $S$.

\section{Construction of the Shell}\label{existence of shell}
In this section we prove the existence (with large probability) of shells in a suitable site percolation model. Later we will apply this fact to show that there exists a shell of rescaled boxes, each having certain good properties as defined in \cref{sec:good boxes}.

Let sites in the lattice $\Z^3$ be independently marked
black with probability $b$ and white otherwise. We wish to
consider paths of a certain type, and we start by defining
two types of steps. An ordered pair of distinct sites
$(x,y)$ is called:
\begin{enumerate}
\item a \df{taxed step} if each non-zero coordinate of $x$ increases in absolute value by $1$ to obtain the corresponding coordinate of $y$, while each zero coordinate of $x$ changes to $-1, 0$ or $1$ to obtain the corresponding coordinate of $y$;
\item a \df{free step} if  $|y| < |x|$ and $y - x \in F$, where $F$ is the set of all vectors obtained by permuting coordinates and flipping signs from any of
\[
(1, 0, 0), (1, 1, 1) \text{ and } (3, 1, 1).
\]
(For example, $(-1,3,1)\in F$.)
\end{enumerate}
Observe that, in a taxed step $(x,y)$, we have $|y|>|x|$.

A \df{permissible path from $x_0$ to $x_\gamma$} is a finite sequence of distinct sites $(x_0, x_1, \ldots, x_\gamma)$ such that for every $i = 1, \ldots, \gamma$, the pair $(x_{i-1}, x_i)$ is either a free step or a taxed step; in the latter case, we also require that $x_i$ is white. 

To obtain a (random) shell $S$ of radius $n$, we let
\begin{equation}\label{eq:Adef}
A = \{y\in \Z^3 : \text{there exists } x \in \Z^3 \text{ with } |x| < n \text{ and a permissible path from } x \text{ to } y\},
\end{equation}
and we define
\begin{equation}\label{eq:Sdef}
S = \{y \in \Z^3 \setminus A : \text{ there exists } x \in A \text{ such that } (x,y) \text{ is a taxed step}\}.
\end{equation}
The remainder of this section is devoted to proving \cref{shell}, which asserts that if the density of black sites is sufficiently high, then a shell of radius $n$ exists with large probability.
\begin{prop} \label{shell}
Let $E_n$ be the event that there exists a shell of radius $n$ consisting of black sites.  There exists $b_1 \in (0,1)$ such that for any $b>b_1$ and $n\ge 1$, we have $\P(E_n)\ge 3/4$.
\end{prop}
Note that the event $E_n$ depends only on the colors of sites in $\{x\in \Z^3 : n\le |x|\le n+ 3\sqrt{n}\}$. However, in proving \cref{shell}, we will show that the set $S$ defined in~\eqref{eq:Sdef} is, in fact, the desired shell with large probability.

\subsection{Probabilistic properties of $S$}\label{random-shell}
In this section we describe random properties of $S$ that hold with probability close to $1$ when the black-site density $b$ is close to $1$. The following lemmas will be used to show that $S$ satisfies properties \ref{S1} and~\ref{S2} with large probability.
\begin{lemma}\label{permissible paths}
There exists a constant $b_2<1$ such that for all $b>b_2$ the following holds. For any integer $k\ge 1$ and any site $x\in\Z^3$, the probability that there exists $y\in\Z^3$ such that $|y|-|x| \ge 0$ and $|y-x| = k$ and there is a permissible path from $x$ to $y$ is at most $2(1-b)^{k/25}$.
\end{lemma}
\begin{proof}
Suppose $\gamma\ge 1$, and consider a (self-avoiding) path $(x_0, x_1, \ldots, x_\gamma)$ such that each step is either taxed or free, and $x_0=x$ and $|x_\gamma| - |x| \ge 0$ and $|x_\gamma-x|=k$. The number of such paths is at most $46^\gamma$, since for each site in $\Z^3$ there are at most $27$ taxed steps and at most $19$ free steps originating at that site. Let $t$ and $f$ be the number of taxed and free steps in this path, and observe that
\[
t+f = \gamma \quad \text{and}\quad 3t -f \ge 0,
\]
since each taxed step can increase the $\ell^1$-norm by at most $3$ and each free step decreases the $\ell^1$-norm by at least $1$. In particular, $t\ge \gamma/4$, and the probability that this path is permissible is $(1-b)^t \le(1-b)^{\gamma/4}$. Furthermore, we in fact have $\gamma\ge \ceil{k/5}$, since each step (taxed or free) has $\ell^1$-norm at most $5$. Choosing $b_2<1$ large enough such that $46(1-b)^{1/4}\le(1-b)^{1/5} \le 1/2$ for all $b>b_2$, the expected number of permissible paths from $x$ to some $y$ with $|y|-|x|\ge 0$ and $|y-x|=k$ is at most
\[
\sum_{\gamma\ge \ceil{k/5}} 46^\gamma (1-b)^{\gamma/4} \le \sum_{\gamma\ge \ceil{k/5}} \left[(1-b)^{1/5}\right]^\gamma \le 2(1-b)^{k/25}.
\]
This completes the proof.
\end{proof}

\begin{lemma}\label{shell_has_caps}
There exists $b_3<1$ such that if $b>b_3$, then for each $n\ge 1$, the set $S$ defined by~\eqref{eq:Adef} and~\eqref{eq:Sdef} satisfies property \ref{S1} with probability at least $7/8$.
\end{lemma}
\begin{proof}
Let $y$ be any site satisfying $|y|=n$ and $\|y\|_\infty \ge n-12$. Observe that for any such $y$, there is a $z$ such that $|z|\le n-1$ and $(z,y)$ is a taxed step; such a $z$ can be found by decreasing the absolute value of each non-zero coordinate of $y$ by $1$ while keeping the sign of each coordinate. For example, if $n\ge 3$ and $y=(0,-n+2,2)$, then $z=(0,-n+3,1)$. Therefore, if there is no $x$ with $|x|\le n-1$ such that there is a permissible path from $x$ to $y$ (so $y$ is not in $A$), then $y\in S$. To this end, the number of sites $x$ such that $|y-x|=k\ge 1$ is at most $4(k+2)^2$, and \cref{permissible paths} implies that if $b>b_2$ and $(1-b)^{1/25}\le 1/2$, then the probability that there exists $x$ such that $|x| \le n-1$ and there is a permissible path from $x$ to $y$ is at most
\[
\sum_{k\ge 1} 4(k+2)^2\cdot  2(1-b)^{k/25} \le (1-b)^{1/25}\sum_{k\ge 1} 8(k+2)^2\cdot  2^{-(k-1)}= C(1-b)^{1/25},
\]
where $1<C<\infty$ is the value of the sum in the middle. Estimating crudely, there are at most $6\cdot 12^3$ sites $y$ satisfying $|y|=n$ and $\|y\|_\infty \ge n-12$. Therefore, taking $b_3\in (b_2,1)$ large enough such that $6\cdot 12^3C(1-b_3)^{1/25} \le 1/8$ finishes the proof.
\end{proof}

\begin{lemma}\label{shell_is_bounded}
There exists $b_4<1$ such that if $b>b_4$, then for each $n\ge 1$, the set $A$ defined by~\eqref{eq:Adef} is finite, and the set $S$ defined by~\eqref{eq:Sdef} satisfies property \ref{S2} with probability at least $7/8$.
\end{lemma}
\begin{proof}
We first verify that if $y\in S$ then $n\le |y|\le n+3\sqrt{n}$ with large probability. The lower bound $|y|\ge n$ is trivial, since $\{x:|x|<n\}\subset A$ and $S\subset \Z^3\setminus A$.  To prove the upper bound, we will show that $A\subset \{y: |y|\le n+3\sqrt{n}-3\}$ with probability at least $7/8$.  Since a taxed step can increase the $\ell^1$-norm by at most $3$, this event implies that $S\subset\{x:|x|\le n+3\sqrt{n}\}$.  Observe that there are at most $8n^3$ sites $x$ with $|x|<n$. If $b>b_2$ and $(1-b)^{1/25}\le 1/2$, then summing over $k\ge 3\sqrt{n}-2$ in \cref{permissible paths} implies that the probability that there exists a permissible path from some $x$ with $|x|\le n-1$ to some $y$ with $|y|\ge n+3\sqrt{n}-3$ (so $|y-x|\ge 3\sqrt{n}-2$) is at most
\[
8n^3 \cdot 4(1-b)^{(3\sqrt{n}-2)/25}.
\]
Since this tends to $0$ as $n\to\infty$, we can choose $b'_4\ge b_2$ large enough such that the above bound on the probability is smaller than $1/8$ for all $n\ge 1$ and all $b>b'_4$.  Therefore, $A\subset \{y : |y|\le n+3\sqrt{n}-3\}$ with probability at least $15/16$.

We now show that $S\subset \{y: \|y\|_\infty \le n\}$ with probability at least $15/16$ for large enough $b$. First, observe that $A\subset \{y : \|y\|_\infty \le n-1\}$ implies that $S\subset \{y: \|y\|_\infty \le n\}$, since each taxed step has $\ell^\infty$-norm $1$, so it suffices to show that there are no permissible paths from some $x$ with $|x|\le n-1$ to some $y$ with $\|y\|_\infty \ge n$. There are at most $60(\ell\wedge (n-\ell))^2$ sites $x$ with $|x|\le n-1$ and $\|x\|_\infty = \ell \le n-1$. For any such $x$, if $\|y\|_\infty \ge n$, then $|y-x|\ge \|y - x\|_\infty \ge n-\ell$. Therefore, if $|x|\le n-1$ and $\|x\|_\infty = \ell \le n-1$, and $b>b'_4$, then summing over $k\ge n-\ell$ in \cref{permissible paths} implies that the probability that there exists a permissible path from $x$ to some $y$ with $\|y\|_\infty\ge n$ is at most
\[
60(n-\ell)^2 \cdot 4 (1-b)^{(n-\ell)/25}.
\]
Summing this expression over $\ell \le n-1$, and again using $(1-b)^{1/25}\le 1/2$ for $b>b'_4$, the probability that there exists a permissible path from some $x$ with $|x|\le n-1$ to some $y$ with $\|y\|_\infty\ge n$ is at most
\[
(1-b)^{1/25} \sum_{k=1}^\infty 240k^2\cdot 2^{-(k-1)}.
\]
Choose $b_4 \ge b'_4$ large enough so that the expression above is smaller than $1/16$. We have shown that if $b>b_4$, then each statement in \ref{S2} holds with probability at least $15/16$, so by the union bound implies, the probability that \ref{S2} holds is at least $7/8$.
\end{proof}


\subsection{Deterministic properties of $S$}
\label{deterministic-shell}
In this section we describe deterministic properties of $S$, as defined in~\eqref{eq:Adef} and~\eqref{eq:Sdef}. Throughout this section we assume that $A$ is bounded, which happens with positive probability by \cref{shell_is_bounded}. We start by showing $S$ satisfies property \ref{S4}.

\begin{lemma}\label{S4 happens}
The set $S$ satisfies property \ref{S4}.
\end{lemma}
\begin{proof}
By symmetry, it suffices to show that $(k,k,k)\in S$ for some $k\in \N$. Let $k = \min\{\ell\in \N: (\ell,\ell,\ell)\notin A\}\ge n$, which is finite, since $A$ is assumed to be bounded. We have $(k-1,k-1,k-1)\in A$ and $(k,k,k)\notin A$, and since $(1,1,1)$ is a taxed step, this implies $(k,k,k)\in S$.
\end{proof}

The next five lemmas essentially say that $S$ varies gradually and cannot contain large flat regions. Taken together, they imply that $S$ satisfies property \ref{S3}.
\begin{lemma}\label{facet protected}
If $x = (x_1, x_2, x_3) \in S$ is such that $x_1\ge 1$ and $x_2\ge 1$ and $x_3\ge 4$, then $x$ is $(2,2,-3)$-protected by some $y\in S$.
\end{lemma}

\begin{proof}
By the definition of $S$ in~\eqref{eq:Sdef}, $x$ must be reachable from $A$ by a taxed step. Since $x$ is not on the spine of $S$, the only site from which we can reach $x$ via a taxed step is $x + (-1,-1,-1)$, so $x + (-1,-1,-1)\in A$. Taking a free step in the $(1, 1, -3)$-direction, this implies $x + (0, 0, -4)\in A$ (here is where we require $x_3\ge 4$). Since $x\in S \subseteq A^c$, we must have $x + (2, 2, -2) \in A^c$, otherwise two free steps in the $(-1,-1,1)$-direction would imply $x\in A$, giving a contradiction.

Now there are two cases. If $x + (1, 1, -3)\in A^c$, then $x + (1, 1, -3)\in S$, since it is a taxed step away from $x+(0,0,-4)\in A$, and we can take $y = x + (1, 1, -3)$. Otherwise, we have $x + (1, 1, -3)\in A$ and $x + (2, 2, -2) \in A^c$, which is a taxed step, so we can take $y=x + (2, 2, -2)$.
\end{proof}

\begin{lemma}\label{1-away protected}
If $x = (x_1, x_2, 1) \in S$ is such that $ x_1\ge 1$ and $x_2\ge 1$, then either $(x_1, x_2, 0)$ or $(x_1+1, x_2+1,0)$ is in $S$.
\end{lemma}

\begin{proof}
Since $x$ is reachable from $A$ by a taxed step and is not on the spine of $S$, we must have $(x_1-1, x_2-1, 0)\in A$. Since $x\in S \subseteq A^c$, we must have $(x_1+1, x_2+1, 0) \in A^c$, otherwise a free move in the $(-1,-1,1)$-direction would imply $x\in A$, giving a contradiction.

There are now two cases. If $(x_1, x_2, 0)\in A^c$, then $(x_1, x_2, 0)\in S$, since it is reachable by a taxed step from $(x_1-1, x_2-1, 0) \in A$ (recall that $0$-coordinates are not required to change along taxed steps). Otherwise, we can take a taxed step from $(x_1, x_2, 0)\in A$ to reach $(x_1+1, x_2+1, 0) \in A^c$, so we take $y = (x_1+1, x_2+1, 0) \in S$.
\end{proof}

\begin{lemma}\label{weakly protected}
If $x = (x_1, x_2, x_3) \in S$ is such that $ x_1\ge 1$ and $x_2\ge 1$ and $x_3\ge 1$, then either $(x_1, x_2, x_3-1)$ or $(x_1+1, x_2+1,x_3-1)$ is in $S$.
\end{lemma}
\begin{proof}
If $x_3=1$, this follows from~\cref{1-away protected}, so assume $x_3\ge 2$. Since $x$ is reachable from $A$ by a taxed step and is not on the spine of $S$, we must have $x + (-1,-1,-1)\in A$. Taking a free move in the $(0,0,-1)$-direction implies that $x+(-1,-1,-2)\in A$. Since $x\in S \subseteq A^c$, we must have $x+(1,1,0)\in A^c$, otherwise free moves in the $(-1,0,0)$ and $(0,-1,0)$-directions imply $x\in A$, giving a contradiction. Similarly, we must have $x+(1,1,-1)\in A^c$, since a free move in the $(-1,-1,1)$-direction brings us back to $x$.

Once again, there are now two cases. If $(x_1, x_2, x_3-1)\in A^c$, then it is in $S$, since it is a taxed step away from $x+(-1,-1,-2) \in A$. Otherwise, $(x_1, x_2,x_3-1) \in A$, which implies $(x_1, x_2, x_3-2)\in A$ by taking a free move in the $(0,0,-1)$-direction. Therefore, we have $x+(1,1,-1) = (x_1+1, x_2+1, x_3-1) \in S$, since it is a taxed step away from $(x_1, x_2, x_3-2)\in A$.
\end{proof}

\begin{lemma}\label{near spine protected}
If $x = (x_1, x_2, x_3)\in S$ is such that $x_1 \ge 1$ and $x_2\ge 1$ and $1\le x_3 \le 3$, then $x$ is $(3,3,-3)$-protected by some $y = (y_1, y_2, 0)\in S$.
\end{lemma}
\begin{proof}
By~\cref{weakly protected}, either $(x_1, x_2, x_3-1)\in S$ or $(x_1+1, x_2+1,x_3-1)\in S$. If $x_3=1$ we are done, otherwise apply~\cref{weakly protected} at most twice more to whichever of these sites is in $S$.
\end{proof}

\begin{lemma}\label{spine protected}
If $x = (x_1, x_2, 0) \in S$ is such that $ x_1\ge 4$ and $x_2\ge 1$, then $x$ is $(-3,3,3)$-protected by some $y \in S$. 
\end{lemma}
\begin{proof}
Since $x\in S$ is reachable from $A$ by a taxed step, we have $(x_1-1, x_2-1, 0) \in A$. Also, we have $(x_1, x_2, 1)\in A^c$, otherwise we could return to $x$ by the free step $(0,0,-1)$. Therefore, $(x_1, x_2, 1)\in S$ because it is reachable from $(x_1-1, x_2-1, 0) \in A$ by a taxed step. \cref{facet protected} implies that there exists $y\in S$ such that $(x_1, x_2, 1)\in S$ is $(-3,2,2)$-protected by $y$. Therefore, $x$ is $(-3,2,3)$-protected by $y$.
\end{proof}

\begin{lemma}\label{all protected}
If $x \in S$ is such that at most one coordinate has absolute value smaller than $4$, then $x$ is protected by $S$.
\end{lemma}

\begin{proof}
If $x = (x_1, x_2, x_3)$ is such that $\min(x_1, x_2, x_3)\ge 4$, then \cref{facet protected} implies $x$ is $(3,3,-3)$-protected by some $y\in S$. By permuting coordinates, $x$ is also $(3,-3,3)$-protected and $(-3,3,3)$-protected by sites in $S$. If $x_1\ge 4$ and $x_2 \ge 4$ and $1 \le x_3 \le 3$, then \cref{facet protected} implies $x$ is $(3,-3,3)$-protected and $(-3,3,3)$-protected by sites in $S$, and~\cref{near spine protected} implies $x$ is $(3,3,-3)$-protected by some site $y$ in the spine of $S$. If $x_1\ge 4$ and $x_2 \ge 4$ and $x_3=0$, then \cref{spine protected} implies that $x$ is $(-3,3,3)$-protected and $(3,-3,3)$-protected by sites in $S$. By symmetry under flipping signs, we also have that $x$ is $(-3,3,-3)$-protected and $(3,-3,-3)$-protected by sites in $S$. Finally, by symmetry under permuting coordinates, if $x$ is in the non-negative octant and has at most one coordinate with absolute value smaller than $4$, then $x$ is protected by $S$, and by symmetry under flipping signs, this holds for all $x\in S$ with at most one coordinate smaller than $4$ in absolute value.
\end{proof}

\subsection{Proof of \cref{shell}}
By \cref{shell_is_bounded,shell_has_caps} and the union bound, if $b_1 = b_3 \vee b_4$, then $b>b_1$ implies that with probability at least $3/4$, the set $A$ given by \eqref{eq:Adef} is bounded and the set $S$ given by \eqref{eq:Sdef} satisfies \ref{S1} and \ref{S2}. On the event that $A$ is bounded, \cref{all protected} implies that $S$ satisfies \ref{S3}, and \cref{S4 happens} implies that $S$ satisfies \ref{S4}. Therefore $S$ is a shell of radius $n$, and $\P(E_n) \ge 3/4$. \qed


\section{Good boxes}\label{sec:good boxes}
Recall the integer $m\ge 1$, which was introduced in \cref{super-sneaky}. Its value will be chosen in the proof of \cref{Z exists} in Section~\ref{all together}; until then, its value is fixed but arbitrary. Let $L = \lfloor \delta/(m^2p) \rfloor$,
where $\delta>0$ is a small constant to be fixed later (in \cref{success prob}).
Also let $M = 36L$. Define the set
\begin{equation}\label{J-definition}
J = (\{0\}^2 \times [-M, M])\ \cup\ (\{0\}\times [-M, M] \times \{0\})\ \cup\ ([-M, M]\times \{0\}^2).
\end{equation}
\begin{figure}
\begin{center}
\raisebox{-0.5\height}{\includegraphics[width=.2\textwidth]{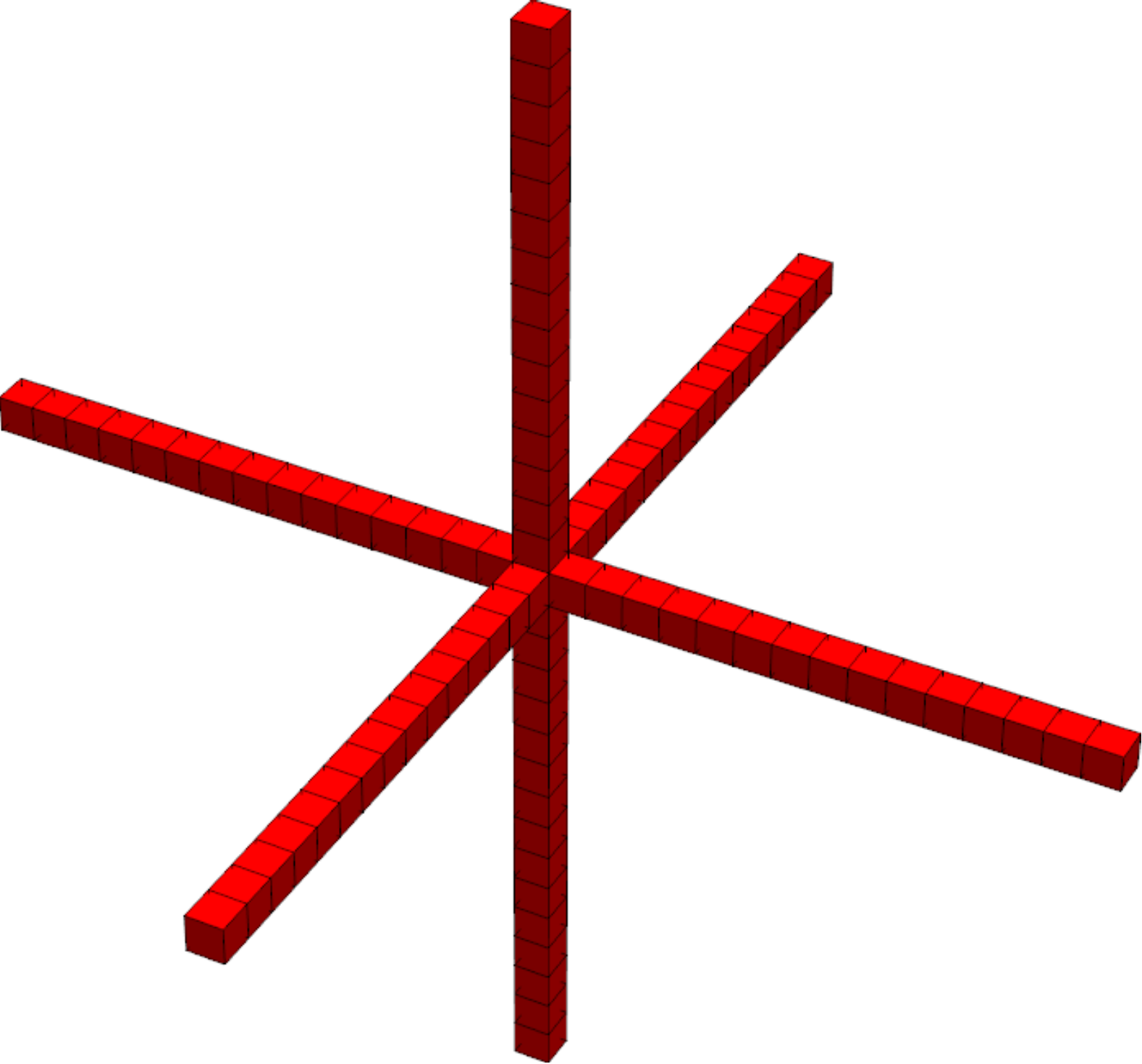}} \hspace{.02\textwidth}
\raisebox{-0.5\height}{\includegraphics[width=.3\textwidth]{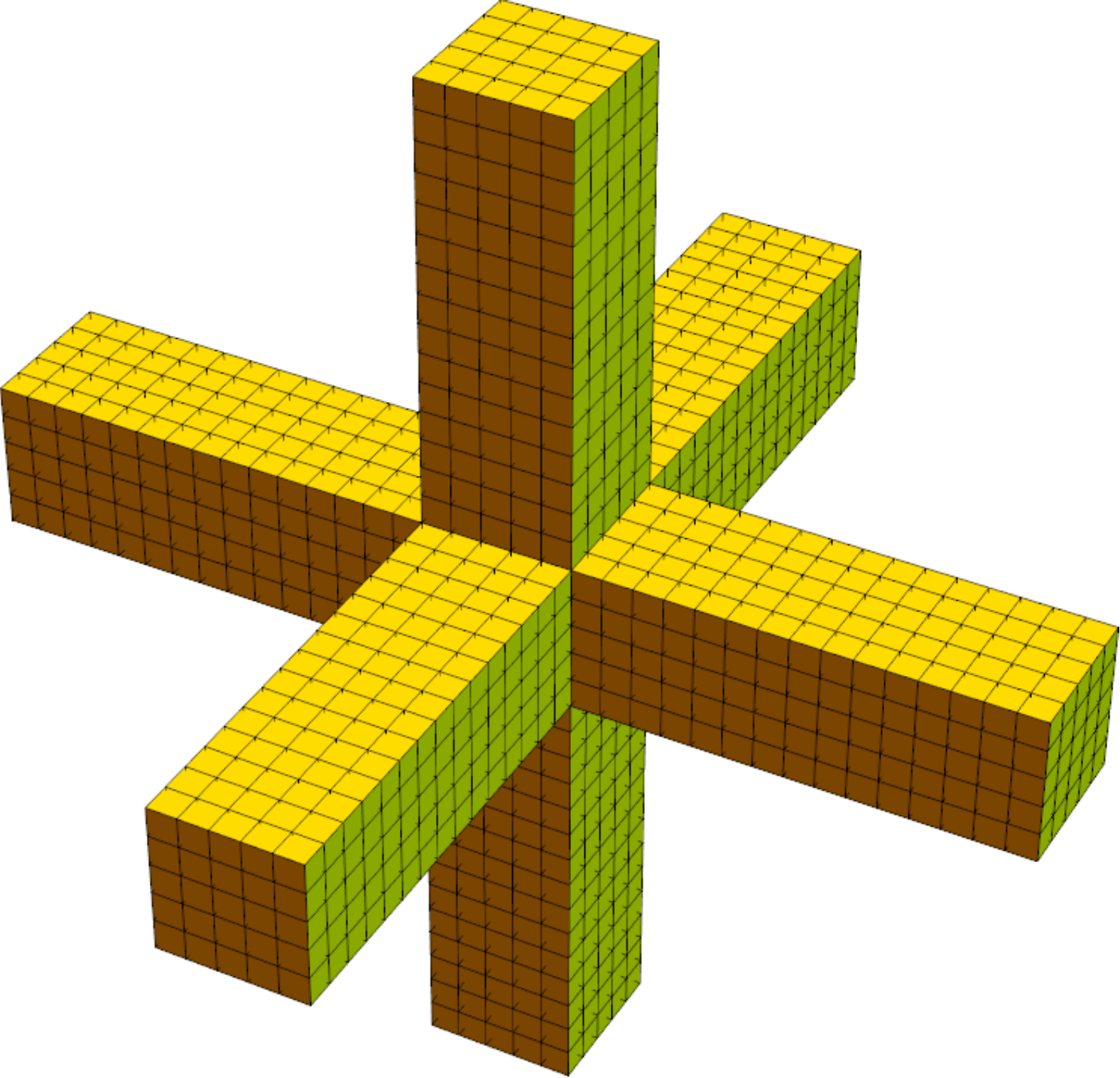}} \hspace{.02\textwidth}
\raisebox{-0.5\height}{\includegraphics[width=.4\textwidth]{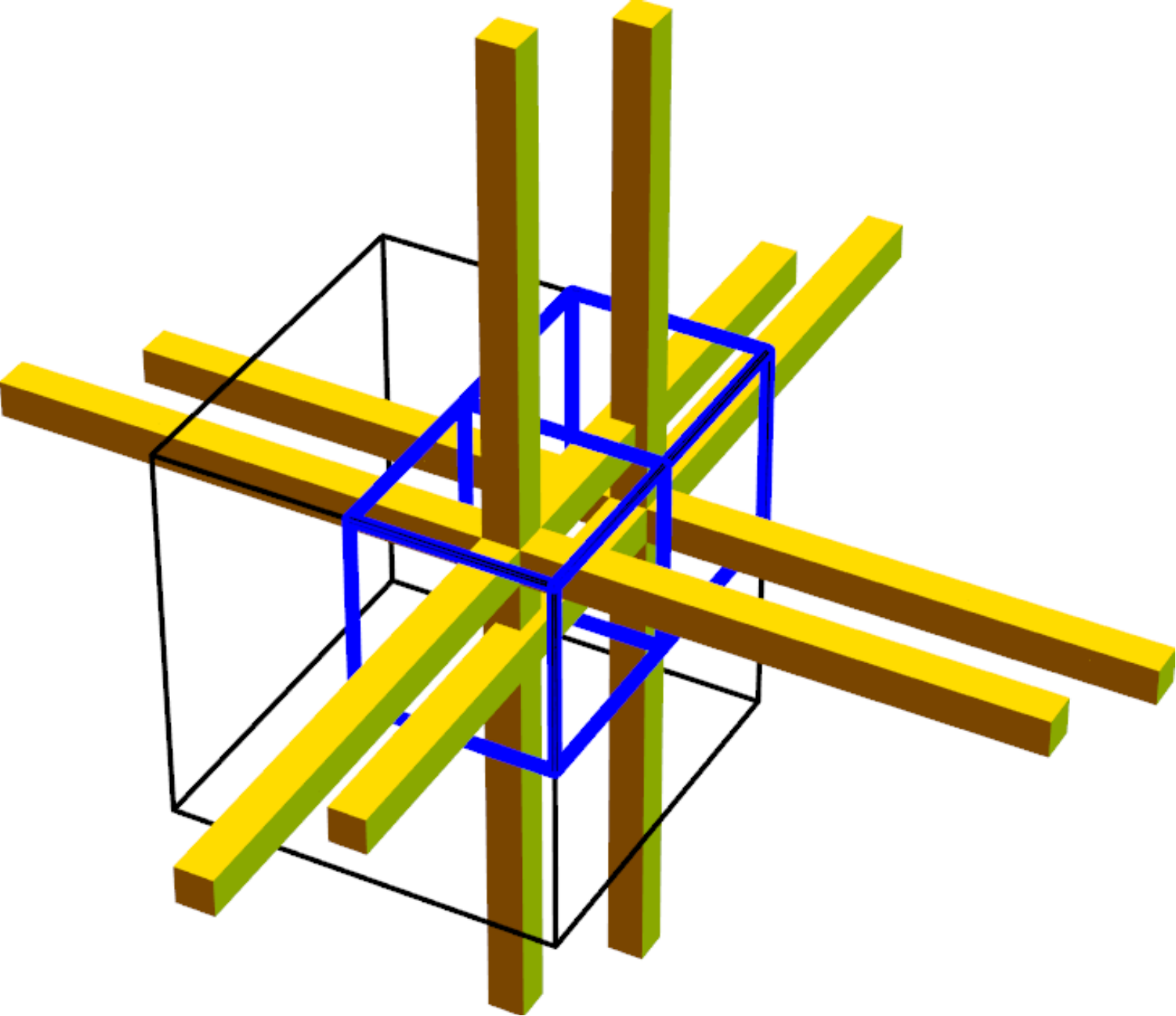}}
\end{center}
\caption{\emph{Left:} the set $J$.  \emph{Middle:} the unoccupied set centered at a nice vertex.  \emph{Right:} a good box: each of the eight subcubes contains a nice vertex (two are shown).}
\label{fig-good}
\end{figure}
Call a vertex $u\in \Z^3$ \df{nice} if $u$ is closed and every vertex within $\ell^\infty$ distance $m$ of the set $u + J$ is initially unoccupied. For each $x\in \Z^3$, define the rescaled box at $x$ to be
\[
Q_x := (2L+1) x + [-L,L]^3.
\]
For $(\sigma_1, \sigma_2, \sigma_3)\in\{+, -\}^3$, define the \df{$(\sigma_1, \sigma_2, \sigma_3)$-subcube} of $Q_x$ to be the set $(2L+1)x + (0,\sigma_1 L]\times(0,\sigma_2 L]\times(0,\sigma_3 L]$. We call a box $Q_x$ \df{good} if each of its eight subcubes contains a nice vertex. See Figure~\ref{fig-good}
for an illustration.

Our goal in this section is to
obtain a lower bound on the probability that a box
is good. We start with a few preliminaries.
We call a vertex $u\in \Z^3$
\df{viable} if
every vertex within distance $m$ of
the set $u + J$ is initially unoccupied.
Note that a viable closed vertex is nice.

We use a simple two-stage procedure to
realize the initial configuration. First, choose closed vertices in $\Z^3$ according to a product measure with density $q$. Second, independently choose \df{active} vertices in $\Z^3$ according to a product measure with density $p/(1-q)$. The set of initially occupied vertices are those that are active but not closed.


\begin{lemma}\label{viable-likely} Assume $q\le 1/2$.
Fix a vertex $u\in \Z^3$ and an $\epsilon>0$. Assume
$\delta\le \epsilon/10^5$. Then the probability that
there is no active vertex within $\ell^\infty$ distance $m$ of $u+J$
is at least $1-\epsilon$. Consequently,
\begin{equation}\label{viable-lb}
\P(u\ \text{\rm is viable})\ge 1-\epsilon.
\end{equation}
\end{lemma}

\begin{proof}
The argument is a simple estimate,
\begin{equation}\label{viable1}
\begin{aligned}
&\P\bigl(\text{there is no active
vertex within $\ell^\infty$ distance $m$ of $u+J$}\bigr)\\
&\ge \bigl[1-p/(1-q)\bigr]^{3(4M+2m+1)(2m+1)^2}\\
&\ge \exp{\bigl[-150\, m^2Mp/(1-q)\bigr]}\\&\ge\exp{(-10800\,\delta)} ,
\end{aligned}
\end{equation}
provided $p$ is small enough. Thus we can choose
any $\delta<\epsilon/(10800)$ to get the probability in
(\ref{viable1}) at least $1-\epsilon$.
\end{proof}

\begin{lemma}\label{good-likely} Fix any $\epsilon>0$,
and assume $\delta=\epsilon/(16\cdot 10^5)$. Then
there exists a constant $C=C(m,\epsilon)$, such that
$q\ge Cp^3$ implies that
the probability that the box $Q_0$ is good is at least
$1-\epsilon$.
\end{lemma}

\begin{proof}
Let $Q'_0$ be the $(+,+,+)$-subcube of $Q_0$. Let $G$ be the event that $Q'_0$ contains a closed vertex.
Provided that $G$ occurs, use any deterministic strategy to
select a closed vertex $u_c\in Q'_0$. By Lemma~\ref{viable-likely},
conditioned on $G$, the probability that there is no
active vertex within distance $m$ of $u_c+J$ is
at least $1-\epsilon/16$, and absence of active vertices implies
absence of occupied vertices. Then
\begin{equation}\label{goodness1}
\begin{aligned}
\P(\text{there is a nice vertex in }Q_0')&\ge
\P(G)\cdot P(u_c\text{ is viable}\mid G) \\
&\ge \left[1-(1-q)^{L^3}\right] \cdot(1-\epsilon/16) \\
&\ge \left[1-\exp(-qL^3)\right]\cdot(1-\epsilon/16)\\
&\ge \left[1-\exp(-(q/p^3)\cdot \delta^3/m^6)\right] \cdot (1-\epsilon/16).
\end{aligned}
\end{equation}

Now choose $C$ large enough
so that the first factor on the last line of
(\ref{goodness1}) exceeds $1-\epsilon/16$. Then (\ref{goodness1})
implies
\begin{equation}\label{goodness2}
\begin{aligned}
\P(\text{there is a nice vertex in }Q_0')\ge 1-\epsilon/8.
\end{aligned}
\end{equation}
Finally, (\ref{goodness2}), symmetry, and
the union bound finish the proof.
\end{proof}


\section{Construction of the stegosaurus}\label{defn of Z}

In this section we construct a set $Z\subset \Z^3$, called the \df{stegosaurus}, which is our candidate for the set satisfying the assumptions of \cref{super-sneaky}. In the next section we will show that this set does, indeed, satisfy these assumptions.

Suppose that there exists a shell $S$ of radius $n$
so that $Q_x$ is good for every $x\in S$. We first fix a set $U\subset \Z^3$, consisting of zero, one or
two nice vertices in each cube $Q_x$ for $x\in S$.
For each $x \in [1,\infty)^3 \cap S$ such that at
least two coordinates of $x$ are at least 4, choose a
nice vertex in the $(-,-,-)$-subcube of $Q_x$.
For each $x \in \{0\} \times [4,\infty)^2 \cap S$
choose a nice vertex in the $(+, +, +)$-subcube of $Q_x$
and another nice vertex in the $(-, +,+)$-subcube of $Q_x$.
We choose nice vertices analogously in the other octants
and coordinate planes. If at least two coordinates of
$x\in S$ are less than 4, we do not choose any nice vertices
from $Q_x$. Let $U$ be the set of all
chosen nice vertices.

A \df{keystone} is a cube of side length
$20L+1$, all eight of whose corners are nice. We now suppose that
there is a keystone centered at each of the six vertices $(\pm n(2L+1),0,0), (0,\pm n(2L+1), 0), (0,0,\pm n(2L+1))$.
Let $K$ be the set of all corner vertices of all keystones ($48$ in all).

We now proceed to define $Z$.
For $u,v \in \Z^3$, let
$B[u,v] = [u_1,v_1]\times [u_2,v_2]\times [u_3, v_3]$
denote a box in $\Z^3$. We will define $Z$ to be the union of certain
boxes, one for each vertex in $U\cup K$.
We define a box $B(u)$ for every $u\in U\cup K$ as follows.
\begin{enumerate}[label=(Z\arabic*)]
\item Suppose $u\in U$ is such that $u\in Q_x $ for a  non-spine
site $x$. Recall that $x$ is protected by $S$, so by
definition it is $a$-protected by some $y_a\in S$, for
each $a$ in
\eqref{eq:protected1}. Such $y_a$ may not be unique;
choose one $y_a$ for each $a$.
Consider two cases.
\begin{itemize}
\item If no  $y_a$ is in the spine, then  let $B(u)=B[0,u]$. (Yellow cuboids in \cref{Zfig}.)
\item If exactly one $y_a$ is in the spine, say, without loss
of generality, with $a=(-3,3,3)$, then let
$v\in U$ be the nice vertex in the $(-,+,+)$-subcube of $Q_{y_a}$.
Take the box $B(u)=B[(v_1,0,0), u]$. (Red cuboids in \cref{Zfig}.)
\end{itemize}
\item For every $u\in U$ such that $u\in Q_x$
where $x$ is on the spine and, say, has first coordinate $0$,
let $B(u)=B[(u_1,0,0),u]$.
Define $B(u)$ similarly if either of the other coordinates of $x$ is $0$. (Green plates in \cref{Zfig}.)
\item For every $u\in K$ (the corner of a keystone),
let $B(u)=B[0,u]$. (Magenta cuboids in \cref{Zfig}.)
\end{enumerate}
Now define the stegosaurus as
\begin{equation}\label{eqn:Zdef}
Z=\bigcup_{u\in U\cup K} B(u).
\end{equation}

\section{Structural properties of the stegosaurus} \label{Z is super-sneaky}

In this section, we will verify that the stegosaurus, $Z$, defined in~\eqref{eqn:Zdef} satisfies the first two assumptions of \cref{super-sneaky}. To do so, we will check that it satisfies the sufficient condition given in \cref{sufficient for sneaky} below.

Throughout this section, in Lemmas \ref{Z is boxes}--\ref{Z cap}, we assume that a shell of good boxes, $S$, of radius $n$ exists, as well as keystones centered at each of the six vertices $(\pm n(2L+1),0,0), (0,\pm n(2L+1), 0), (0,0,\pm n(2L+1))$, and that the sets $U$ and $K$ are given as in the previous section.

Recall that $\nonbrs(u)$ denotes the number of coordinates in which $u$ has a neighbor outside of $Z$. For a set of vertices $B\subset \Z^3$ (which will typically be a box or union of boxes), we say $v$ is a \df{corner}  of $B$ if $v\in B$ has neighbors outside $B$ in all three coordinates, and $v$ is on an \df{edge} of $B$ if $v\in B$ has neighbors outside $B$ in exactly two coordinates. We start with a simple observation that follows from the construction of $Z$.

\begin{lemma}\label{Z is boxes} Suppose $Z$ is defined as in~\eqref{eqn:Zdef}.
\begin{enumerate}[label= \textup{(\roman*)}]
\item If $w\in Z$ has $\nonbrs(w) = 3$, then $w$ is a corner of $B(u)$ for some $u\in U\cup K$.
\item If $w\in Z$ has $\nonbrs(w)=2$, then $w$ is either a corner or on an edge of $B(u)$ for some $u\in U\cup K$.
\end{enumerate}
\end{lemma}
\begin{proof}
In the case $\nonbrs(w)=3$, if $w\in B(u)$ is not a corner of $B(u)$, then this gives a contradiction. In the case $\nonbrs(w)=2$, if $w\in B(u)$ is not a corner and not on an edge of $B(u)$, then this gives a contradiction.
\end{proof}

The next lemma states that vertices near the coordinate axes are not corner or edge vertices of $Z$. This is because these vertices are protected by the keystones, in the sense that vertices of $Z$ near the axes lie within the union of cuboids $B(u)$ for $u\in K$. Note that the vertices in the statement of \cref{Z protects axes} exclude those that are corners or edges of the set  $\cup_{u\in K} B(u)$.

\begin{lemma}\label{Z protects axes} Suppose $Z$ is defined as in~\eqref{eqn:Zdef}.
\begin{enumerate}[label= \textup{(\roman*)}]
\item If $w\in Z$ has two coordinates that are \emph{strictly} smaller than $10L$ in absolute value, then $\nonbrs(w)\le 1$.
\item If $w\in Z$ has one coordinate equal to $10L$ in absolute value and one coordinate \emph{strictly} smaller than $10L$ in absolute value and $\|w\|_\infty < (2L+1)n+10L$, then $\nonbrs(w)\le 1$.
\end{enumerate}
\end{lemma}
\begin{proof}
Observe that $\cup_{u\in K} B(u)\subset Z$ contains all vertices $v\in \Z^3$ such that $\|v\|_\infty \le (2L+1)n+10L$ and at least two coordinates of $v$ are smaller than or equal to $10L$ in absolute value. Property \ref{S2} of the shell $S$ implies that $\|v\|_\infty \le (2L+1)n + 10L$ for all $v\in Z$, so if $w\in Z$ has at least two coordinates smaller than or equal to $10L-1$ in absolute value, then $\nonbrs(w)\le 1$ (with equality if and only if $\|w\|_\infty = (2L+1)n + 10L$). This proves part (i). For part (ii), observe that if $w$ satisfies the three conditions given, then $w$ has at most one neighbor outside of $\cup_{u\in K} B(u)$, so $\nonbrs(w)\le 1$.

\end{proof}

The next lemma gives sufficient conditions for $Z$ to satisfy the first two assumptions in~\cref{super-sneaky}, and by~\cref{Z is boxes}, it suffices to verify these conditions for each box comprising $Z$. Essentially, it says that for each $u\in U\cup K$, we need to show that all edges and corners of $B(u)$ that are not sufficiently close to $u$ are hidden within other cuboids or plates.

\begin{lemma}\label{sufficient for sneaky}
Let $Z$ be defined as in~\eqref{eqn:Zdef}. Suppose that for each $u\in U\cup K$, if $w\in B(u)$ is a corner or on an edge of $B(u)$, then either
\begin{enumerate}[label= \textup{(\roman*)}]
\item \label{ss1} $w$ has at least two coordinates strictly smaller than $10L$ in absolute value, or\\
$w$ has one coordinate equal to $10L$ in absolute value and one coordinate strictly smaller than $10L$ in absolute value and $\|w\|_\infty<(2L+1)n+10L$, or
\item \label{ss2} $w=u$, or
\item \label{ss3} $u\in U$ and $w$ is on an edge (not a corner) of $B(u)$ and $w\in u+J$, or
\item \label{ss4} $u\in K$ and $w$ is either a corner or on an edge of $B(u)$ and $w\in u+J$, or
\item \label{ss5} there exists $v \in U$ such that $w\in B(v)$ and $w$ is not a corner and not on an edge of $B(v)$.
\end{enumerate}
Then $Z$ satisfies the first two assumptions of \cref{super-sneaky}.
\end{lemma}

\begin{proof}
Suppose $w\in Z$ has $\nonbrs(w)\ge 2$. By \cref{Z is boxes}, $w$ is a corner or on an edge of $B(u)$ for some $u\in U\cup K$, and by \cref{Z protects axes}, $w$ does not satisfy \ref{ss1}. If $w$ satisfies condition \ref{ss5} for some $v\in U$, then $w$ can only have neighbors that are outside of $B(v)\subset Z$ in at most one coordinate, so $\nonbrs(w)\le 1$; this is a contradiction. Therefore, by \ref{ss2}, \ref{ss3} and \ref{ss4}, we have $w\in u+J$. Since $u$ is a nice vertex, there are no initially occupied vertices within distance $m$ of $u+J\ni w$. This verifies that $Z$ satisfies the second assumption in \cref{super-sneaky}. If, in addition, $\nonbrs(w)=3$, then \cref{Z is boxes} implies $w$ is a corner of $B(u)$ for some $u\in U\cup K$.  If $u\in U$, then \ref{ss2} and \ref{ss3} imply that $w=u$. In this case, $w\in U$ is a nice vertex, so $w$ is closed. Suppose now that $u\in K$. In this case, we claim that if $\nonbrs(w)=3$ and $w\in B(u)$, then $w=u$. Indeed, if $w\in \cup_{u\in K} B(u)$ has $\nonbrs(w)=3$, then $w\in K$ (in fact, $w$ must be one of the $24$ outermost corners of the keystones, so $w$ is a permutation of one of the points $(\pm((2L+1)n+10L),\pm10L,\pm10L)$). Furthermore, if $u$ and $u'$ are distinct points in $K$ with $\nonbrs(u)=\nonbrs(u')=3$, then $B(u)\cap B(u')$ is disjoint from $K$. Therefore, if $w\in B(u)$ and $\nonbrs(w)=3$, then $w=u$. This verifies that $Z$ satisfies the first assumption in~\cref{super-sneaky}.
\end{proof}


In the next series of lemmas, we verify the conditions in \cref{sufficient for sneaky} hold for the corner and edge vertices of every box used to construct $Z$. There are four cases for the location of $u\in U\cup K$, and by symmetry of the construction of $S$ and $Z$, we will assume for all statements that $u$ is in the first octant or near one of the positive coordinate planes or axes. The analogous statements (for other octants) hold by permuting coordinates and flipping signs.

We start with cuboids of vertices far from the coordinate axes; this case corresponds to the yellow cuboids in \cref{Zfig}.

\begin{lemma}\label{Z bulk}
Suppose $u\in U$ is such that $u\in Q_x$ with $\min(x_1, x_2, x_3)\ge 4$. Then $u$ satisfies the condition of \cref{sufficient for sneaky}.
\end{lemma}
\begin{proof}
Since $x$ is far from the spine, it is $a$-protected by a non-spine site $y_a\in S$ for each $a$ in~\eqref{eq:protected1}, so $B(u) = B[0,u]$. We will show that all corner and edge vertices of $B(u)$ satisfy some condition of \cref{sufficient for sneaky}.

First observe that each point in the intervals
\begin{equation} \label{eq:bulk1}
\begin{aligned}
&[(u_1-M)\vee 1,u_1] \times\{u_2\}\times\{u_3\},\\
&\{u_1\} \times[(u_2-M)\vee 1,u_2]\times\{u_3\},\\
\text{and} \quad &\{u_1\}\times \{u_2\}\times [(u_3-M)\vee 1,u_3]
\end{aligned}
\end{equation}
satisfy either condition \ref{ss2} or \ref{ss3} of \cref{sufficient for sneaky}, depending on whether the vertex is a corner or on an edge of $B(u)$ (note that $u$ is the only corner among these vertices).
Taking $a=(-3,3,3)$, we have that there is a $v\in U$ such that $v\in Q_{y_a}$, which implies $v_2>u_2$ and $v_3>u_3$ and $0<u_1-v_1<  4(2L+1) < M$. Since $B[0,v]\subseteq B(v)$, this implies that the intervals 
\begin{equation} \label{eq:bulk2}
\begin{aligned}
&[0,(u_1-M)\vee 1) \times \{u_2\}\times\{u_3\},\\
&\{0\} \times [1,u_2] \times \{u_3\},\\
\text{and} \quad&\{0\} \times \{u_2\} \times [1,u_3]
\end{aligned}
\end{equation}
are contained in $B(v)$, and do not contain any  corner or edge vertices of $B(v)$.  By taking $a$ to be $(3,-3,3)$ or $(3,3,-3)$, similar arguments imply that the intervals
\begin{equation} \label{eq:bulk3}
\begin{aligned}
& \{u_1\} \times [0,(u_2-M)\vee 1) \times\{u_3\}, &  & \{u_1\}  \times\{u_2\}\times [0,(u_3-M)\vee 1),\\
&[1,u_1]  \times \{0\} \times \{u_3\}, & & [1,u_1] \times \{u_2\} \times \{0\}, \\
&\{u_1\} \times \{0\} \times [1,u_3], & \text{and} \qquad& \{u_1\}\times [1,u_2]\times \{0\}
\end{aligned}
\end{equation}
are each contained in $B(v)$ for some $v\in U$ (depending on $a$), and do not contain any  corner or edge vertices of $B(v)$. Therefore, the vertices in~\eqref{eq:bulk2} and~\eqref{eq:bulk3} satisfy condition \ref{ss5} of \cref{sufficient for sneaky}.
Finally, the intervals
\begin{equation} \label{eq:bulk4}
\begin{aligned}
& [0,u_1]\times\{0\}\times\{0\},\\
&\{0\}\times [0,u_2]\times\{0\},\\
\text{and} \quad&\{0\}\times\{0\}\times [0,u_3]
\end{aligned}
\end{equation}
satisfy condition \ref{ss1} of \cref{sufficient for sneaky}. The vertices in~\eqref{eq:bulk1}, \eqref{eq:bulk2}, \eqref{eq:bulk3} and \eqref{eq:bulk4} comprise all of the corner and edge vertices of $B(u)$, so this completes the proof of the lemma.
\end{proof}

The next case is for vertices near enough to the coordinate planes to have their cuboids protected by plates; this case corresponds to the red cuboids in \cref{Zfig}.

\begin{lemma}\label{Z near spine}
Suppose $u\in U$ is such that $u\in Q_x$ and $x_1\in [1,3]$ and $\min(x_2, x_3)\ge 4$. Then $u$ satisfies the condition of \cref{sufficient for sneaky}.
\end{lemma}
\begin{proof}
If $x$ is $a$-protected by a non-spine site $y_a\in S$ for each $a$ in~\eqref{eq:protected1}, then $B(u) = B[0,u]$, and the proof that $u$ satisfies the condition of \cref{sufficient for sneaky} is identical to the proof of \cref{Z bulk}. Otherwise, $x$ is $(-3,3,3)$-protected by the spine site $y = (0, y_2, y_3)\in S$ (recall that in this case we allow $y_2=x_2$ or $y_3=x_3$), and $x$ is $a$-protected by the (non-spine) sites $y_a\in S$ for $a= (3,-3,3)$ and $a=(3,3,-3)$. Since $x$ is a non-spine site in the positive octant, we have that $u$ is in the $(-,-,-)$-subcube of $Q_x$. Since $y \in \{0\} \times [4,\infty)^2 \cap S$, there exists $v\in U$ such that $v$ is in the $(-,+,+)$-subcube of $Q_y$, and we have $B(u) = B[(v_1, 0,0),u]$ with $-L\le v_1<0$.  We will show that all corner and edge vertices of $B(u)$ satisfy some condition of \cref{sufficient for sneaky}.

First, since $0<u_1 - v_1 < 4(2L+1) < M$, the vertices in the intervals
\begin{equation} \label{eq:nearspine1}
\begin{aligned}
&[v_1+1,u_1] \times\{u_2\}\times\{u_3\},\\
&\{u_1\} \times[(u_2-M)\vee 1,u_2]\times\{u_3\},\\
\text{and} \quad&\{u_1\}\times \{u_2\}\times [(u_3-M)\vee 1,u_3]
\end{aligned}
\end{equation}
each satisfy condition \ref{ss2} or \ref{ss3} of \cref{sufficient for sneaky}. Now observe that since $v\in Q_y$ and $y$ is on the spine, we have $B(v) = B[(v_1, 0, 0), v]$.  Furthermore, since $v$ is in the $(-,+,+)$-subcube of $Q_y$ and $u$ is in the $(-,-,-)$-subcube of $Q_x$, we have $v_2> u_2$ and $v_3>u_3$, so the intervals
\begin{equation} \label{eq:nearspine2}
\begin{aligned}
&\{v_1\} \times [0,u_2] \times\{u_3\},\\
\text{and} \quad&\{v_1\} \times\{u_2\} \times [0,u_3]
\end{aligned}
\end{equation}
are contained in $B(v)$, and do not contain any corner or edge vertices of $B(v)$. Therefore, these vertices satisfy condition \ref{ss5} of \cref{sufficient for sneaky}.

 Next, let $a=(3,-3,3)$, and recall $y_a$ is not on the spine. Therefore, taking $v_a\in Q_{y_a} \cap U$, we have $B[0,v_a]\subseteq B(v_a)$, and $v_{a1}>u_1$ and $v_{a3}>u_3$ and $0< u_2 - v_{a2} < 4(2L+1)<M$. These inequalities (and the analogous argument for $a=(3,3,-3)$) imply that the intervals
\begin{equation} \label{eq:nearspine3}
\begin{aligned}
& \{u_1\} \times [0,(u_2-M)\vee 1) \times\{u_3\}, &  & \{u_1\}  \times\{u_2\}\times [0,(u_3-M)\vee 1),\\
&[1,u_1]  \times \{0\} \times \{u_3\}, & & [1,u_1] \times \{u_2\} \times \{0\}, \\
&\{u_1\} \times \{0\} \times [1,u_3], & \text{and} \qquad& \{u_1\}\times [1,u_2]\times \{0\}
\end{aligned}
\end{equation}
are each contained in $B(v_a)$ (for the respective value of $a$), and do not contain any  corner or edge vertices of $B(v_a)$. Therefore, the vertices in \eqref{eq:nearspine3} satisfy condition \ref{ss5} of \cref{sufficient for sneaky}.

Finally, the intervals
\begin{equation} \label{eq:nearspine4}
\begin{aligned}
& [v_1,u_1]\times\{0\}\times\{0\}, & & \\
&\{v_1\}\times [0,u_2]\times\{0\}, & & [v_1, 0] \times \{u_2\} \times \{0\}, \\
&\{v_1\}\times\{0\}\times [0,u_3], &\quad  \text{and} \qquad &  [v_1, 0] \times \{0\} \times \{u_3\}
\end{aligned}
\end{equation}
consist entirely of vertices having at least two coordinates smaller than or equal to $L$ in absolute value, and therefore satisfy condition \ref{ss1} of \cref{sufficient for sneaky}. Note that the two intervals on the right in~\eqref{eq:nearspine4} extend below the $yz$-coordinate plane; this compensates for the fact that the non-spine cuboids providing protection to $B(u)$ may not extend (far enough) below the $yz$-coordinate plane (their contribution is in the middle row of \eqref{eq:nearspine3}). The vertices in \eqref{eq:nearspine1}, \eqref{eq:nearspine2}, \eqref{eq:nearspine3} and \eqref{eq:nearspine4} comprise all of the corner and edge vertices of $B(u)$, so this completes the proof.
\end{proof}

The next lemma addresses plates, depicted in green in \cref{Zfig}.

\begin{lemma}\label{Z spine}
Suppose $u\in U$ is such that $u$ is in the $(+,+,+)$-subcube of $Q_x$ with $x = (0,x_2, x_3)$ and $\min(x_2, x_3)\ge 4$.  Then $u$ satisfies the condition of \cref{sufficient for sneaky}.
\end{lemma}
\begin{proof}
Since $x$ is on the spine, we have that $B(u) = B((u_1, 0, 0), u)$, and that $x$ is $a$-protected by $y_a\in S$ for $a=(3,-3,3)$ and $a=(3,3,-3)$. (We note that $x$ is also $(-3,-3,3)$-protected and $(-3,3,-3)$-protected by sites in $S$, and these sites are needed to protect the $(-,+,+)$-subcube of $Q_x$, but are not needed here. Also, note that in this case $B(u)$ is a rectangle, not a cuboid, so only has four edges and corners.) First, observe that each vertex in the intervals
\begin{equation} \label{eq:spine1}
\begin{aligned}
&\{u_1\} \times[(u_2-M)\vee 1,u_2]\times\{u_3\},\\
\text{and} \quad&\{u_1\}\times \{u_2\}\times [(u_3-M)\vee 1,u_3]
\end{aligned}
\end{equation}
satisfy either conditions \ref{ss2} or \ref{ss3} of \cref{sufficient for sneaky}, depending on whether it is a corner or on an edge of $B(u)$.
Now, let $a = (3,-3,3)$, and observe that $y_a$ is not on the spine. This follows from the definition of $a$-protected, and the fact that $(y_{a})_2\ge x_2-3>0$ and $(y_{a})_3\ge x_3>0$, so $(y_{a})_1>x_1=0$. Therefore, taking $v_a\in Q_{y_a} \cap U$, we have $B[0,v_a]\subseteq B(v_a)$, and $(v_{a})_1>u_1$ and $(v_{a})_3>u_3$ and $0< u_2 - (v_{a})_2 < 4(2L+1)<M$. These inequalities (and the analogous argument for $a=(3,3,-3)$) imply that the intervals
\begin{equation} \label{eq:spine2}
\begin{aligned}
& \{u_1\} \times [0,(u_2-M)\vee 1) \times\{u_3\}, &  & \{u_1\}  \times\{u_2\}\times [0,(u_3-M)\vee 1),\\
&\{u_1\} \times \{0\} \times [1,u_3], & \text{and} \qquad& \{u_1\}\times [1,u_2]\times \{0\}
\end{aligned}
\end{equation}
are contained in $B(v_a)$ (for the respective value of $a$), and do not intersect any corner or edge vertices of $B(v_a)$. Therefore, these vertices satisfy condition \ref{ss5} of \cref{sufficient for sneaky}. Finally, the vertices in the intervals
\begin{equation} \label{eq:spine3}
\begin{aligned}
&\{u_1\}\times [0,u_2]\times\{0\},\\
\text{and} \quad&\{u_1\}\times\{0\}\times [0,u_3]
\end{aligned}
\end{equation}
have at least two coordinates with absolute values smaller than or equal to $L$, and therefore satisfy condition \ref{ss1} of \cref{sufficient for sneaky}. The vertices in \eqref{eq:spine1}, \eqref{eq:spine2} and \eqref{eq:spine3} comprise all of the corner and edge vertices of $B(u)$, so this completes the proof.
\end{proof}

Finally, we show that the keystones (magenta in \cref{Zfig}) are protected by the rest of $Z$.

\begin{lemma}\label{Z cap}
Suppose $u\in K$. If $p$ is small enough (depending on $\delta$ and $m$) such that $L>12$, then $u$ satisfies the conditions of \cref{sufficient for sneaky}.
\end{lemma}
\begin{proof}
By symmetry, we may assume $u = ((2L+1)n+10L, 10L, 10L)$. Observe that the point $u' = ((2L+1)n-10L, 10L, 10L)\in K$ has $B(u')\subset B(u)$, so it suffices to consider the outer corner, $u$, of the keystone. Here (in the first line below) is where we use the full radius of the set $u+J$, which is $M = 36L$. Observe that the vertices in the intervals
\begin{equation}\label{eq:keystone1}
\begin{aligned}
&[(2L+1)n - 26L, (2L+1)n+10L] \times \{10L\} \times \{10L\},\\
&\{(2L+1)n+10L\} \times [0,10L] \times \{10L\}, \\
\text{and} \quad&\{(2L+1)n+10L\} \times \{10L\} \times [0,10L]
\end{aligned}
\end{equation}
all satisfy condition \ref{ss4} of \cref{sufficient for sneaky}. Now, by property \ref{S1} of the shell $S$, the site $y = (n-12,6,6)$ is in $S$, so there exists $v\in Q_y \cap U$. Since $y$ is far from the spine (all coordinates are at least $4$), we have $B(v)=B[0,v]$. Furthermore, for $p$ small enough such that $L>12$,
\[
v_1\ge (2L+1)(n-12)-L = (2L+1)n - (25L+12) \ge (2L+1)n - 26L,
\]
and $v_2 \ge 6(2L+1)-L \ge 11L$ and $v_3 \ge 11L$. Therefore, all of the vertices in the intervals
\begin{equation}\label{eq:keystone2}
\begin{aligned}
&[0,(2L+1)n - 26L] \times \{10L\} \times \{10L\},\\
& \{0\} \times [1,10L] \times \{10L\},\\
\text{and} \quad&\{0\}  \times \{10L\} \times [1,10L]
\end{aligned}
\end{equation}
are contained in $B(v)$ and do not intersect any corner or edge vertices of $B(v)$. Therefore, the vertices in~\eqref{eq:keystone2} satisfy condition \ref{ss5} of \cref{sufficient for sneaky}. Finally, all of the vertices in the intervals
\begin{equation}\label{eq:keystone3}
\begin{aligned}
&[0, (2L+1)n+10L) \times \{0\} \times \{10L\}, &\qquad & \{(2L+1)n+10L\} \times [0,10L) \times \{0\}, \\
&[0, (2L+1)n+10L) \times \{10L\} \times \{0\}, && \{(2L+1)n+10L\} \times \{0\}\times [0,10L) ,\\
&[0, (2L+1)n+10L] \times \{0\} \times \{0\}, &&\{0\} \times [0,10L] \times \{0\}, \\
\text{and} \quad& \{0\} \times \{0\}\times [0,10L] &&
\end{aligned}
\end{equation}
satisfy condition \ref{ss1} of \cref{sufficient for sneaky}.  The vertices in~\eqref{eq:keystone1}, \eqref{eq:keystone2} and~\eqref{eq:keystone3} comprise all of the corner and edge vertices in $B(u)$, so this completes the proof.
\end{proof}

\section{Putting it all together}\label{all together}
In this section, we put together the pieces from previous sections to prove the existence of a set $Z$ satisfying \cref{super-sneaky} with probability tending to $1$ as $p\to 0$. This entails identifying a shell of radius $n$ for some $n$ and adding keystones. However, since all $6$ keystones appear with probability about $q^{48}$, we will need to construct polynomially many (in $1/p$) shells before finding one that can successfully be adorned with keystones to complete the construction of $Z$. It then remains to check that this $Z$ satisfies the third condition of \cref{super-sneaky}, which states that threshold $\thr=2$ modified bootstrap percolation, restricted to $Z$ and without closed vertices, does very little. We start with this verification, then move on to actually identifying $Z$.

\subsection{Threshold 2 bootstrap percolation}
\label{threshold-2-bp}
In the next lemma, we assume that $q=0$, so the initial state has no closed vertices. We also suppose that the dynamics follow the threshold $r=2$ modified bootstrap rule~\eqref{modified} \df{internal} to the box $[-N,N]^3$, meaning we set all vertices outside of $[-N,N]^3$ to be initially (and forever) empty. By monotonicity, any vertex left unoccupied in the final configuration by these dynamics will also be left unoccupied by the threshold $r\ge 2$ modified bootstrap dynamics internal to $Z\subset [-N,N]^3$ with $q\ge 0$. In what follows, we say a set $R\subset \Z^3$ is \df{internally spanned} if the set $R$ is fully occupied in the final configuration by the threshold $r=2$ modified bootstrap dynamics internal to $R$.

\begin{lemma}\label{threshold 2}
Set $q=0$. Fix an integer $s>0$, and let $N=\lfloor p^{-s}\rfloor$.
Suppose the dynamics follow the modified threshold $r=2$ bootstrap rule~\eqref{modified} {internal} to $[-N,N]^3$.
Then all connected clusters (maximal connected sets) of occupied vertices in the final configuration
are cuboids.
Furthermore, with probability converging to $1$
as $p\to 0$, the final configuration has the following
two properties: all fully occupied
cuboids have side lengths at most $6s$, and the origin is not occupied.
\end{lemma}
\begin{proof}
The first claim follows from the bootstrap rule.
To demonstrate the second claim, fix an integer
$k>0$. Let $E_k$ be the
event that the
final configuration contains an occupied cuboid whose longest
side has length at least $k$. If $E_k$ occurs, $[-N,N]^3$
contains an
internally spanned cuboid $R$ whose longest side length
is in the interval $[k/2,k]$ \cite{AL}. Then, any plane
perpendicular to the longest
side of $R$ that intersects $R$ must contain an
occupied vertex within $R$. There are at most $(2N+1)^3 k^3$ possible
selections of the cuboid $R$. Therefore,
\begin{equation}\label{threshold2-1}
\P(E_k)\le 10\,N^3 k^3 (k^2p)^{k/2}= 10\,k^{k+3}p^{(k-6s)/2}.
\end{equation}
Furthermore,
\begin{equation}\label{threshold2-2}
\begin{aligned}
&\P(\text{the origin is occupied in the final configuration})\\
&\le \P\bigl([-k,k]^3\text{ contains an initially occupied
vertex}\bigr)+\P(E_k)\\
&\le (2k+1)^3p+\P(E_k).
\end{aligned}
\end{equation}
If $k>6s$, then the probabilities in (\ref{threshold2-1})
and (\ref{threshold2-2}) both go to $0$ as $p\to 0$.
\end{proof}

\subsection{Existence of $Z$} \label{existence-Z}
Let $T = \floor{p^{-146}/10}$, and define a sequence of integers $(n_0, \ldots, n_T)$ by $n_0 = \floor{p^{-292}}$ and $n_T = 2n_0$ and for $k=1, \ldots, T-1$,
\begin{equation}\label{eq:akdef}
n_k = n_0 + k \floor{5\sqrt{n_T}}.
\end{equation}

For each $k$, we will attempt to find a shell $S$ of radius $n_k$ within the annulus
\[
\mathcal{A}_k = \left\{x\in\Z^3 : n_k \le |x| \le n_k + 3\sqrt{n_k}\right\}
\]
such that $Q_x$ is a good box for every $x\in S$; let $\mathtt{Shell}(k)$ denote the event that there is such a shell. Simultaneously, for each $k$, we also attempt to find keystones (boxes with side lengths $20L+1$, all of whose corners are nice) centered at the six vertices $(\pm n_k(2L+1),0,0), (0,\pm n_k(2L+1), 0), (0,0,\pm n_k(2L+1))$. Let $K_k$ be the set of corner vertices of these $6$ (potential) keystones, so $|K_k| = 48$, and let $\mathtt{Keystones}(k)$ be the event that all of the vertices in $K_k$ are nice. We declare the $k^\text{th}$ attempt to be successful if we find both the shell of radius $n_k$ and the keystones at all six locations, and we let $F_k = \mathtt{Shell}(k) \cap \mathtt{Keystones}(k)$ denote the event that the $k^\text{th}$ attempt is successful. First, we show that the annuli $(\mathcal{A}_k)_{k\ge 0}$ are spaced sufficiently far apart.

\begin{lemma}\label{ak bound}
For the sequence $(n_0, \ldots, n_T)$ defined in~\eqref{eq:akdef} and $p<1/2$, we have
\[
n_0 \le n_{k+1} \le n_T=2n_0  \qquad \text{and} \qquad n_{k+1} - (n_k + 3\sqrt{n_k}) \ge 10^3
\]
for all $0\le k\le T-1$.
\end{lemma}
\begin{proof}
The first lower bound $n_k\ge n_0$ is obvious. For the upper bound,
\[
n_k \le n_0 + T\cdot 5\sqrt{2n_0} \le n_0 + (p^{-146}/10)\cdot (5\sqrt{2} p^{-146}),
\]
which is smaller than $2n_0$ for $p<1/2$. The second bound follows from $\floor{5\sqrt{n_T}} - 3\sqrt{n_k} \ge 2\sqrt{n_T}-1 \ge 10^3$ for $p<1/2$.
\end{proof}

Now we are ready to prove the following lemma, which is the key to finding the set~$Z$.

\begin{lemma}\label{success prob}
There exist $\delta>0$ and $C>0$ such that if $q\in [Cp^3,1/2]$, then for all sufficiently small $p$ the events $\{F_k\}_{k\ge 0}$ are independent and $\P(F_k) \geq p^{145}$ for every $k\ge 0$.
\end{lemma}
\begin{proof}
For each $k\ge 0$, let
\[
\mathcal{B}_k = \Bigl\{u\in \Z^3 : (n_k - 100)(2L+1) \le |u| \le (n_k + 3\sqrt{n_k} + 100)(2L+1)
\Bigr\},
\]
and note that the sets $\mathcal{B}_k$ for $k\ge 0$ are disjoint (for small enough $p$) by the second inequality in \cref{ak bound}. By the definition of a good box, for each $x\in \Z^3$, the event that $Q_x$ is a good box depends only on the states (occupied, closed or empty) of vertices in the set
\begin{equation}\label{eq:dependence}
Q_x + [-M-m,M+m]^3 \subset (2L+1)x + [-40L,40L]^3,
\end{equation}
where the containment holds for small enough $p$, recalling $M=36L$. Therefore, since $(2L+1)\mathcal{A}_k + [-40L,40L]^3 \subset \mathcal{B}_k$, the event $\mathtt{Shell}(k)$ depends only on the states of the vertices in $\mathcal{B}_k$. Also, for each $k$, the event $\mathtt{Keystones}(k)$ depends only on the states of vertices in
\begin{align*}
&K_k + [-M-m,M+m]^3 \subset\\
&\Bigl\{u\in \Z^3 : n_k(2L+1) - 3(10L+M+m) \le |u| \le n_k(2L+1) + 3(10L+M+m)\Bigr\},
\end{align*}
which is a subset of $\mathcal{B}_k$ for all small enough $p$. Therefore, for each $k$, the event $F_k$ depends only on the states of vertices in $\mathcal{B}_k$, so the events $\{F_k\}_{k\ge 0}$ are independent.

If we paint each site $x\in \Z^3$ black if $Q_x$ is a good box, and white otherwise, then the argument surrounding~\eqref{eq:dependence} shows that this coloring forms a $120$-dependent random field. Let $b_1$ be the constant from \cref{shell}. By~\cite{LSS}, if $\P(Q_x \text{ is good})\ge 1-\epsilon$ for all $x\in \Z^3$ and $\epsilon>0$ is sufficiently small, then there exists $b>b_1$ such that this random coloring stochastically dominates a product measure with density $b$ of black sites. Therefore, choosing such an $\epsilon\in (0,1/50)$ and letting $\delta = \epsilon/(16\cdot 10^5)$, \cref{good-likely} implies the existence of $C>0$ such that $\P(Q_x \text{ is good})\ge 1-\epsilon$ whenever $q\ge Cp^3$, and \cref{shell} implies that $\P(\mathtt{Shell}(k))\ge 3/4$ for every $k\ge 0$.

To produce the keystones in the $k^{\text{th}}$ annulus, as in Section~\ref{sec:good boxes}, we use a two-stage procedure to
realize the initial state: we choose closed vertices
through a product measure with density $q$ and
independently choose active
vertices through a product measure with density $p/(1-q)$,
then declare a vertex initially occupied if it is active but not closed.  By \cref{viable-likely}, for each $u\in K_k$, the probability that there are no active vertices in $u+J + [-m,m]^3$ is at least $1-\epsilon>49/50$. Therefore, the probability that there are no active vertices in $K_k+J+ [-m,m]^3$ is at least $2/50$. Now, independently, the vertices in $K_k$ are all closed with probability $q^{48} \ge C^{48} p^{144}$, so
\[
\P(\mathtt{Keystones}(k)) \ge (C^{48}/25) p^{144}
\]
for every $k\ge 0$.

Finally, if we identify the initial states of the vertices as $(\text{occupied}, \text{empty}, \text{closed}) = (-1,0,+1)$, then both $\mathtt{Shell}(k)$ and $\mathtt{Keystones}(k)$ are increasing events. Indeed, if $\mathtt{Shell}(k)$ occurs, then there is a shell $S$ of radius $n_k$ contained in $\mathcal{A}_k$ such that $x\in S$ implies $Q_x$ is good. By flipping the states of some vertices from occupied to empty or closed, or from empty to closed, we cannot change a good box to a bad box, so we cannot destroy the shell $S$. Likewise, if $\mathtt{Keystones}(k)$ occurs, we cannot alter its occurrence by flipping vertices from occupied to empty or closed, or from empty to closed.  Therefore, by the FKG inequality we have
\[
\P(F_k) \ge \P(\mathtt{Shell}(k)) \cdot \P(\mathtt{Keystones}(k)) \ge  (3C^{48}/100) p^{144},
\]
which is at least $p^{145}$ for sufficiently small $p$.
\end{proof}

We can now prove existence of the set $Z$ in the desired region.

\begin{lemma}\label{Z exists}
Let $\delta>0$ and $C>0$ be chosen as in \cref{success prob}, and let $N_0 = \floor{L/3}\floor{p^{-292}}$  If $q\in [Cp^3,1/2]$, then with probability converging to $1$ as $p\to 0$ there exists a set $Z$ such that the initial configuration on $Z$ satisfies the three assumptions of \cref{super-sneaky}, and such that $[-N_0,N_0]^3 \subset Z \subset [-22N_0,22N_0]^3$.
\end{lemma}
\begin{proof}
By \cref{success prob}, recalling $T = \floor{p^{-146}/10}$,
\[
\P\Bigl(\cap_{k=0}^{T-1} F_k^c\Bigr) \leq (1-p^{145})^{T-1} \le
\exp\bigl(-p^{-1}/10 + 2\bigr) \to 0 \quad \text{as $p\to 0$}.
\]
Therefore, with high probability the events $\mathtt{Shell}(k)$ and $\mathtt{Keystones}(k)$ occur for some $k\le T-1$. Given a shell, $S$, of radius $n_k$ comprised of good boxes and keystones centered at the six appropriate vertices, we can define $Z$ as in~\eqref{eqn:Zdef}. By \cref{sufficient for sneaky,Z bulk,Z near spine,Z spine,Z cap} and symmetry, this set $Z$ satisfies the first two assumptions of \cref{super-sneaky}. By \cref{ak bound}, we have $n_k + 3\sqrt{n_k} \le n_T = 2n_0 = 2\floor{p^{-292}}$. Therefore, for all small enough $p$, if $u\in Z$, then
\[
\|u\|_\infty \le (2L+1) n_T + L \le 7L n_0 \le 22 N_0
\]
This shows $Z\subset [-22N_0,22N_0]^3$. Note that $22N_0 \le  p^{-300}$ for small enough $p$. Therefore, applying \cref{threshold 2} with $s=300$ implies that $Z$ satisfies the third assumption of \cref{super-sneaky}, and that we may take $m = 12s = 3600$.

Property \ref{S4} of the shell $S$ of radius $n_k$ implies that there exists $\ell \ge n_k/3$ such that $(\ell,\ell,\ell)\in S$. Since this site is far from the spine of $S$, there exists a nice vertex $u\in Q_{(\ell,\ell,\ell)}$ so that
\[
\bigl[0,(2L+1)(n_k-1)/3\bigr]^3 \subset B[0,u]\subset Z.
\]
Considering the other seven diagonal directions, symmetry and \cref{ak bound} imply that $[-N_0,N_0]^3\subset Z$, which finishes the proof.
\end{proof}


\subsection{Lack of percolation}\label{lack-of-percolation}

We are now ready to conclude the proof of
Theorem~\ref{three-modified}.

\begin{proof}[Proof of Theorem~\ref{three-modified}]
Part (i) was proved in Section~\ref{low-q},
so we proceed to prove (ii).
To prove that, provided $q>Cp^3$, the
final density goes to $0$,
we apply Lemma~\ref{Z exists}, Proposition~\ref{super-sneaky},
and Lemma~\ref{threshold 2}. The last step is to show that,
almost surely,
$\perc$  does not happen.

Let $N_0$ be as in Lemma~\ref{Z exists},
and $N=\lfloor N_0/2\rfloor$.
We say that $x\in \Z^3$ is \df{$N$-closed}, if
$Z$ satisfies the three assumptions of
Proposition~\ref{super-sneaky} and $Nx+[-N_0,N_0]^3
\subseteq Z\subseteq  Nx+[-22N_0, 22N_0]^3$.
A site is \df{$N$-open} if it is not $N$-closed.
Observe that, for
an $N$-closed site $x$, there is no nearest neighbor path
between  $Nx+[-N,N]^3$ and $(Nx+[-22N_0, 22N_0]^3)^c$,
on eventually occupied vertices (provided $p$ is small enough).

By Lemma~\ref{Z exists} and translation invariance, there are constants $\delta>0$ and $C>0$ so that
for any $\alpha>0$, there exists $p^*>0$ such that
the probability that a fixed
site is $N$-closed is at least $1-\alpha$ for all $p<p^*$.
Observe that $x,y\in \Z^3$
at $\ell^\infty$ distance at least $100$ are $N$-closed
independently. Using \cite{LSS}, we therefore may choose
$p^*$ small enough to guarantee
the event that there is no infinite connected set of
$N$-open sites has probability $1$. This event is a subset of
$\perc^c$, as we will now argue.

Assume $\pi$ is an infinite self-avoiding
nearest-neighbor path of vertices starting at the origin.
If two neighboring
vertices $z_1,z_2\in \Z^3$ satisfy $z_i\in Nx_i+[-N,N]^3$
for $i=1,2$ and $x_1\ne x_2$, then $x_1$ and $x_2$
are also neighbors
in $\Z^3$.
It follows that $\pi$
must enter $Nx+[-N,N]^3$,
for some $N$-closed site $x$. But then $\pi$ must also
exit $Nx+[-22N_0, 22N_0]^3$, and therefore it
cannot consist solely
of eventually occupied vertices. Therefore $\perc$ does
not happen.
\end{proof}

 \section{The standard model and big obstacles}\label{standard model}

In this section, we sketch the proofs of Theorems \ref{three-standard} and \ref{larger obstacles}.
We do not give full details, as the arguments are very similar to
those for Theorem~\ref{three-modified}, and in the case of \cref{three-standard}
we do not obtain a precise value of the critical exponent.

We start with the standard bootstrap percolation with density $q$ of closed vertices and density $p$ of initially occupied vertices, \cref{three-standard}\ref{three-standard-ii}, and the proof for the standard bootstrap percolation with big obstacles, \cref{larger obstacles}\ref{large obstacles-ii}, will be similar.  As already mentioned,
the key difference from the modified bootstrap percolation is that for the standard model the construction of the set $Z$ from
a shell consisting of
nice boxes does not suffice: while such $Z$
still enjoys the same protection
in the bulk, it is vulnerable to the invasion of occupied
vertices near the spine. That is,
a vertex in a plate near a coordinate plane can be penetrated by two occupied
neighbors outside of the plate (and one additional occupied vertex in the plate) so
Lemma~\ref{Z spine} no longer holds.
The plate therefore
needs to be replaced by a cuboid. To protect
the exposed corners of this cuboid, we need two
closed vertices on the same axis-parallel line. This problem, which we are
unable to overcome, necessitates $q$ to be on the order of $p^2$
for the critical probability for existence of a successful stegosaurus.

To make the first new ingredient precise, we say that a box
$Q_x$ is \df{swell} if for each of its eight subcubes, and
for each of the three coordinate directions, there is a
line in that direction that contains two nice vertices in
that subcube. Lemma~\ref{good-likely} is then replaced by
the following.

\begin{lemma}\label{swell-likely} Fix any $\epsilon>0$,
and assume $\delta=\epsilon/10^7$. Then
there exists a large enough constant $C=C(m,\epsilon)$, such that
$q\ge  Cp^2$ implies that
the probability that the box $Q_0$ is swell is at least
$1-\epsilon$.
\end{lemma}

\begin{proof}
By monotonicity, we may assume that $q\le p^{3/2}$.
Let $Q_0'$ be the $(+,+,+)$-subcube of $Q_0$, and
let $G_c$ (resp.~$G_n$) be the event that $Q_0'$
contains two closed (resp.~nice) vertices
on the same vertical (i.e., $z$-axis-parallel) line.
Provided that $G_c$ occurs, select such
a pair $(u_c,u_c')$ of closed
vertices by any deterministic strategy. By Lemma~\ref{viable-likely},
conditioned on $G_c$, the probability that there is no
active vertex within $\ell^\infty$ distance $m$ of $(u_c+J)\cup (u_c'+J)$ is
at least $1-\epsilon/48$. Thus,
\begin{equation}\label{swellness1}
\begin{aligned}
\P(G_n)
&\ge
\P(G_c)\cdot \P(u_c\text{ and }u_c'\text{ are both viable}\mid G_c) \\
&\ge \P(G_c)\cdot(1-\epsilon/48).
\end{aligned}
\end{equation}
To get a lower bound on the probability of $G_c$,
divide $Q_0'$ into two halves of height $\lfloor L/2\rfloor$
by a horizontal cut
and then observe $G_c$ will occur if a vertical line has a
closed vertex in each half. Therefore,
\begin{equation}\label{swellness2}
\begin{aligned}
\P(G_c)&\ge 1-\left[1-(1-(1-q)^{\lfloor L/2\rfloor})^2\right]^{L^2}\\
&\ge 1-\exp\left(-[L(1-e^{-qL/3})]^2\right)\\
&\ge 1-\exp\left(-[L^2q/4]^2\right),
\end{aligned}
\end{equation}
for $p$ small enough. Therefore, we can choose
$C$ large enough to make $\P(G_c)\ge 1-\epsilon/48$
and therefore, by (\ref{swellness1}),
$\P(G_n)\ge 1-\epsilon/24$. The union bound
then ends the proof.
\end{proof}

\begin{proof}[Sketch of the proof of Theorem~\ref{three-standard}\ref{three-standard-ii}]
Provided $q>Cp^2$, for a large enough $C$,
Lemma~\ref{swell-likely} now ensures the existence of a
shell with the same properties as in the modified case.

The second difference is the definition of $Z$. Now we add to $U$
additional nice vertices in the boxes corresponding to
the spine of the shell.
Namely, for each $x \in \{0\} \times [4,\infty)^2 \cap S$, we now
choose {\it a pair\/} of nice vertices,
that lie on a line in the $x$-direction,
in the $(+, +, +)$-subcube of $Q_x$
and another such pair in the $(-, +,+)$-subcube of $Q_x$.
(Again, we make analogous choices in the other octants
and coordinate planes.) Selection of other nice vertices is identical.
Every $x$ in the spine now contributes two boxes that are
defined by the collinear pairs and are no longer plates of width $1$.
That is, the second part (Z2) of the construction
is replaced by the following.
\begin{itemize}
\item[(Z2')] For every $u,u'\in U$ such that $u, u'$ are in the same
subcube of $Q_x$
where $x$ is on the spine and, say, has first coordinate $0$,
let $B(u,u')=B[(u_1,0,0),u']$.
Define $B(u,u')$ similarly if either
of the other coordinates of $x$ is $0$.
\end{itemize}

Then the verification of protection properties
in Section~\ref{Z is super-sneaky} is the same, except for the outward
edges of $B(u,u')$, which are of length at most $L$ and
connect two closed vertices.

The final ingredient that is slightly different is
Lemma~\ref{threshold 2}
on the threshold $r=2$ rule, where the bound on the side
length of final cuboids is $12s$ instead of $6s$,
again using the standard argument from \cite{AL}.

With these adjustments, the proof proceeds along the same
lines, and the constants can be chosen in the same order.
\end{proof}

The proof of \cref{larger obstacles} follows nearly the same argument, and we now sketch the proof for this case.
\begin{proof}[Sketch of the proof of \cref{larger obstacles}]
The proof for \cref{larger obstacles}\ref{large obstacles-i} follows the same argument given in \cref{low-q} for \cref{three-modified}\ref{three-modified-i}. It suffices to consider the modified bootstrap percolation dynamics, but now for a site $x\in\Z^3$ to be $N$-open, there must not be any obstacle centers within $\ell^1$ distance $1$ of $Nx+[0,N)^3$. This makes the collection of $N$-open sites a $1$-dependent random field. \cref{theta3-Lopen} still holds in this case, and an application of \cite{LSS} handles the dependence. The rest of the proof is nearly identical to the proof of \cref{three-modified}\ref{three-modified-i}.

For \cref{larger obstacles}\ref{large obstacles-ii} it suffices to consider the standard bootstrap dynamics, and the modifications required in this case are similar to those for \cref{three-standard}\ref{three-standard-ii} given above. Again, plates are no longer impervious to invasion by occupied vertices, and we need two nice vertices on the same axis-parallel line within each box for the construction of $Z$, so we require a shell of swell sites to exist. However, now that obstacle centers appear with density $q$ and make closed each of their six neighboring vertices, the probability that the box $Q_x$ is swell is almost the same as the probability that the box $Q_x$ is good. Therefore, \cref{swell-likely} holds for $q>Cp^3$, and its proof is almost identical to the proof of \cref{good-likely}. The rest of the proof is the same as for \cref{three-standard}\ref{three-standard-ii}.
\end{proof}

\pagebreak
\section{Open Problems} \label{open problems}

We conclude by adding a few open problems to the collection
in \cite{GH}.

\begin{enumerate}
\item Consider
the standard model on $\Z^3$ with
threshold $\thr=3$, and suppose $p,q\to 0$
in such a way that $\log q/\log p\to \alpha<3$. Does
the final density then go to $0$?
\item Consider the modified model on $\Z^d$ with
threshold $\thr=d$,
and $q=p^d(\log p^{-1})^{-\gamma}$. For which
$\gamma>0$ does the final density approach $0$ (resp.~1),
as $p\to 0$? The answer is already unknown in $d=2$ and
$d=3$, although \cite{GM} and Theorem~\ref{three-modified}
provide some information. For $d\ge 4$, not even the
power scaling $p^d$ is established. (See the open
problem (ii) in Section 6 of \cite{GH}.)
\item Consider the modified model on $\Z^3$ with
threshold $\thr=3$.
Let $T$ be the first time the origin is occupied and let
$\lambda=\pi^2/6$. When $q=0$, \cite{Hol2} proved that
\begin{equation*}\label{T-scale}
\P\Bigl[\exp\exp\left((\lambda-\epsilon)p^{-1}\right)
\le T \le \exp\exp\left((\lambda+\epsilon)p^{-1}\right)\Bigr]\to 1,
\end{equation*}
for any $\epsilon>0$.  Does this still hold if
we assume instead that $p,q\to 0$ in such a way that
$q<p^\gamma$ for some $\gamma>3$? (The lower bound on $T$
is immediate.)

\end{enumerate}

\section*{Acknowledgements}

JG was partially supported by the NSF grant DMS--1513340, Simons Foundation Award \#281309, and the Republic of Slovenia's Ministry of Science program P1--285. DS was partially supported by NSF grant DMS--1418265. JG and DS also gratefully acknowledge the hospitality of the Theory Group at Microsoft Research, where some of this work was completed.

%

\end{document}